\documentclass[a4paper,11pt]{article}
\usepackage{Lutsko_Style_Balint}
\usepackage[T1]{fontenc}
\begin{document}

\title{Invariance Principle for the Random Wind-Tree Process}

\author{
{\sc Christopher Lutsko$^*$ and B\'alint T\'oth$^{*\dagger}$}
\\[8pt]
$^*$University of Bristol, UK
\\
$^\dagger$R\'enyi Institute, Budapest, HU
}

\maketitle

\begin{abstract}
  \noindent Consider a point particle moving through a Poisson distributed array of cubes all oriented along the axes - the random wind-tree model introduced in Ehrenfest-Ehrenfest (1912) \cite{ehrenfest-ehrenfest_12}. We show that, in the joint Boltzmann-Grad and diffusive limit this process satisfies an invariance principle. That is, the process converges in distribution to a Brownian motion in a particular scaling limit. In a previous paper (2019) \cite{lutsko-toth_19} the authors used a novel coupling method to prove the same statement for the random Lorentz gas with spherical scatterers. In this paper we show that, despite the change in dynamics, the same strategy with some modification can be used to prove an invariance principle for the random wind-tree model.

\medskip\noindent
{\sc MSC2010: } 60F17; 60K35; 60K37; 60K40; 82C22; 82C31; 82C40; 82C41

\medskip\noindent
{\sc Key words and phrases:} wind-tree; Ehrenfest model; invariance principle; scaling limit; coupling; exploration process

\end{abstract}


\section{Introduction}
\label{s:Introduction}

In this paper we consider the motion of a point particle through an array of randomly placed, identically oriented cubes in $\R^3$ - the so-called random wind-tree process \cite{ehrenfest-ehrenfest_12}. In a recent paper \cite{lutsko-toth_19} the authors showed that the random Lorentz gas (i.e the same process with the cubes replaced by spheres \cite{lorentz_05}) satisfies an invariance principle in a particular scaling limit which is intermediate between the kinetic and purely diffusive time scales. In this paper we prove an invariance principle for the wind-tree process in a similar intermediate regime. The proof will follow similar lines. However there are two key differences: in the Lorentz gas, after collision with a randomly placed scatterer (in $3$ dimensions) the velocity is redistributed independently of the initial velocity while for the wind-tree process the velocities form a genuine Markov chain. On the other hand as the collisions are simpler in the wind-tree setting the necessary geometric estimates follow with less effort.

More formally let $\cP$ be a Poisson point process of intensity $\varrho>0$ in $\R^3$ (our results hold for general dimension $d \ge 3$, however to reduce notation we restrict to $d=3$). Let $\cQ_r$ be a cube of side length $r$ oriented parallel with the axes and let $\cP+\cQ_r$ be an array of \emph{obstacles/scatterers}. We consider the trajectory of a point particle $X^{r,\varrho}(t)$ starting at the origin ($X^{r,\varrho}(0)=0$) with a fixed initial velocity of unit length. The particle then flies in straight lines, reflecting elastically off of the obstacles. In this setting the origin is in $(\cP+\cQ_r)^c$ with probability tending to $1$, hence such a trajectory is well-defined (see {\cite[(2)]{lutsko-toth_19}} for more details).

A fundamental open problem for both the random wind-tree model and the random Lorentz gas is to prove an invariance principle in the diffusive limit. That is, in the limit

\begin{equation}\label{diffusive limit}
  \frac{X^{r,\varrho}(Tt)}{\sqrt{T}} \qquad , \qquad T \to \infty,
\end{equation}
does the scaled process converge weakly to a Wiener process? In our previous paper we showed that the Lorentz gas satisfies an invariance principle in the limit \eqref{diffusive limit} if we \emph{simultaneously} take the low-density limit in a particular scaling limit. The aim for this paper is to replicate that result for the wind-tree model.


\begin{figure}[ht!]
  \begin{center}    
    \includegraphics[width=0.9\textwidth]{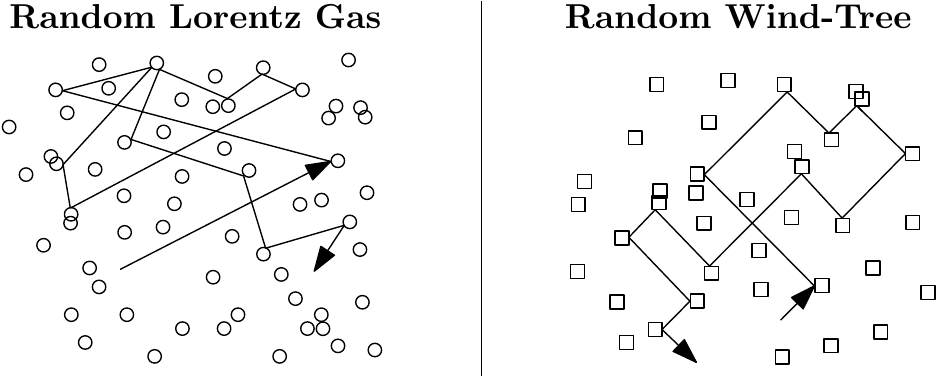}
  \end{center}
  \caption{%
    {\tt Typical trajectories of the random wind-tree and Lorentz gas models. Note the difference in the dynamics: in the wind-tree model the velocities are restricted to a finite set (in 2 dimensions there are only 4 possible velocities). While in the Lorentz gas the velocities are uniformly distributed on the sphere.}   
  }%
  
  \label{fig:rLG_vs_rWT}

\end{figure}

\subsection{Scaling and Main Result}
\label{ss:Scaling and Main Result}

Fix a probability vector $\mathbf{p} = (p_1 ,p_2, p_3)$ with $p_i>0$ for all $i$, and let $\abs{\mathbf{p}} = \sqrt{p_1^2+p_2^2+p_3^2}$. The state-space of velocities is then

\begin{equation}
  \Omega := \left\{ v \in S_1^{2}: |v_i| = \frac{p_i}{\abs{\mathbf{p}}} \right\}
\end{equation}
Fix the initial velocity $\dot{X}^{r,\varrho}(0^+) \in \Omega$. We study the process $t \mapsto X^{r,\varrho}(t)$ on $[0,T]$ in the joint Boltzmann-Grad and diffusive scaling limit:

\begin{gather}
  \begin{gathered} \label{BG-Diff-limit} 
    r   \to       0\qquad,    \qquad r^{2}\varrho  \to       \abs{\mathbf{p}}^{-1} \qquad, \qquad T(r)       \to       \infty  \\
  t          \mapsto   \frac{X(tT)}{\sqrt{T}},
  \end{gathered}
\end{gather}
note that $\abs{\mathbf{p}}^{-1}$ is the cross-sectional area of the cube as viewed by the particle, and we have dropped the dependence on $r$ and $\varrho$ in the notation (thus $X^{r,\varrho}(t) =X(t)$). With that, the main result of this paper is the following invariance principle:


\begin{theorem}\label{thm:main theorem}
  Consider the intermediate scaling limit \eqref{BG-Diff-limit} such that $\lim_{r\to 0}T(r)r^{2} =0$ then

  \begin{equation}
    \left\{t \mapsto T^{-1/2}X(tT) \right\} \Longrightarrow \left\{t \mapsto W(t) \right\}
  \end{equation}
  as $r \to 0$ in the averaged-quenched sense (see below). Where $W(t)$ is a Wiener process with covariance matrix $M = \operatorname{diag}(v_1^2,v_2^2,v_3^2)$ in $\R^3$.

\end{theorem}

The proof follows from a joint construction of $t \mapsto X(t)$ and a second Markovian process which we introduce in Section \ref{ss:Markovian Flight Process}. In Section \ref{ss:Main Technical Result} we state and outline the proof of the main technical theorem of the paper (Theorem \ref{thm:main-technical}). Theorem \ref{thm:main theorem} is then a straightforward corollary of that theorem.

\begin{remark}
For the Lorentz gas we proved the same theorem with the asymptotic constraint $\lim_{r\to 0} T(r)r^2 \abs{\log r }^2 =0$. The reason for this logarithmic correction are those collisions for which the angle between incoming and outgoing velocities is small. In the wind-tree model the velocity of the point particle is restricted to a fixed discrete set, hence the log factor can be removed.
\end{remark}

In this context there are two relevant limits one could take:

\begin{enumerate}

\item[(Q)]
\emph{Quenched limit}: 
For almost all (i.e. typical) realizations of the underlying Poisson point process, with  averaging over the random initial velocity of the particle.

\item[(AQ)]
\emph{Averaged-quenched} (a.k.a. \emph{annealed}) \emph{limit}: 
Averaging over the random initial velocity of the particle \emph{and} the random placement of the scatterers.

\end{enumerate}
This paper (and our previous paper \cite{lutsko-toth_19}) are in the averaged-quenched setting.

\subsection{Related work}
\label{ss:History}

While we cannot hope to give an exhaustive account we present here some of the related work. The wind-tree model was introduced in the famous monograph by Paul and Tatiana Ehrenfest {\cite[Appendix to Section 5]{ehrenfest-ehrenfest_12}} as a simplified model to explain the return to equilibrium of the velocity distribution of a gas. It is noteworthy that, in defining the model P. and T. Ehrenfest considered randomly placed scatterers oriented along the axes (as we have done), however the periodic wind-tree model (often referred to as the Ehrenfest model) - where \emph{rectangular} scatterers are centered at the points of a hypercubic lattice - is the better studied model. This owes to the fact that the periodic setting can be studied using methods from parabolic dynamical systems. While the random wind-tree model is less well-understood, it remains of interest as a stochastic process and a model for diffusion in particle systems.

\textbf{2D Periodic Wind-Tree:} The periodic wind-tree model \emph{in 2 dimensions} has been the focus of a lot of recent research. In this setting the billiard flow is parabolic (i.e close orbits diverge polynomially). Thus (unlike for the periodic Lorentz gas - see for example the survey \cite{marklof_14}) the tools of hyperbolic dynamics cannot be used. Instead the standard approach is to use the so-called Katok-Zemliakov construction (see \cite{tabachnikov_95}), which allows one to replace the billiard flow by linear flow on translation surfaces.

There have not yet been any theorems concerning the diffusive limit or an invariance principle for the periodic wind-tree process. However there have been a number of interesting and contrasting results concerned with the speed of diffusion and exceptional trajectories. Hardy and Weber \cite{hardy-weber_80} showed that some specific directions diffuse at a rate of $\log T \log \log T$. While Delecroix-Hubert-Leli\`{e}vre \cite{delecroix-hubert-lelievre_14} showed that typical (with respect to angle) trajectories satisfy the superdiffusive polynomial diffusion rate $T^{2/3}$. Additionally Delecroix \cite{delecroix_13} showed that for any rectangular scatterer, there is a set of diverging trajectories with positive Hausdorff measure. While Hubert-Leli\`{e}vre-Troubetzkoy \cite{hubert-lelievre-troubetzkoy_11}  and then Avila and Hubert \cite{avila-hubert_17} showed that the billiard flow is recurrent for almost every direction. Finally Fr\k{a}czek and Ulcigrai \cite{fraczek-ulcigrai_14} proved that generically the billiard flow is not ergodic. 

\textbf{Random Wind-Tree and Lorentz gas:} At the moment there are fewer rigorous results about the random wind-tree model and the random Lorentz gas than their periodic counterparts. Gallavotti \cite{gallavotti_69}, \cite{gallavotti_70} used classical (probabilistic) methods to show that in the (annealed) Boltzmann-Grad limit (i.e \eqref{BG-Diff-limit} with $T$ constant) both models obey a linear Boltzmann equation with different collision terms. For a wide class of Lorentz gas models with spherically symmetric scattering potentials, Spohn \cite{spohn_78} and Boldrighini-Bunimovich-Sinai (for the hard-core random Lorentz gas) \cite{boldrighini-bunimovich-sinai_83} showed that in the Boltzmann-Grad limit the Lorentz gas converges in the annealed, respectively quenched sense to a Markovian flight process. To our knowledge all the previous work on these random models has been in the Boltzmann-Grad limit and for \emph{finite} time intervals. The holy grail - the invariance principle in the diffusive limit - remains completely open for both models.

More recently Marklof and Str\"{o}mbergsson \cite{marklof-strombergsson_19} prove convergence to a limiting transport process for a wide class of spherically symmetric potentials and scatterer configurations. In particular this approach subsumes these previous results on the random Lorentz gas \cite{gallavotti_69}, \cite{gallavotti_70}, \cite{spohn_78}, \cite{boldrighini-bunimovich-sinai_83} and covers many other cases (periodic or quasi-crystals) all with spherically symmetric scattering potentials.

\section{Coupling Construction}
\label{s:Coupling Construction}

\subsection{State-Space and Notation}
\label{ss:Collisions}

Returning now to the random wind-tree model, for the rest of the paper we assume the initial velocity is fixed to be $v_0\in \Omega$. This will aid in the exposition but can be assumed without loss of generality, since the time taken to reach this velocity is exponentially bounded.

At each collision one component of the velocity changes sign. Let $\vartheta_i:\R^3\to \R^3$ be such that $\vartheta_i(v)_j = (-1)^{\delta_{i,j}}v_j$ for $j=1,2,3$. During a collision the probability $\probab{v \mapsto \vartheta_i(v)} = p_i$. For any $v \in \Omega$ let $\Omega_{v}$ denote the set of accessible velocities after one collision starting from $v$, namely

\begin{equation}
  \Omega_v = \{ w \in \Omega : w = \vartheta_i(v) \mbox{ for some } 1 \le i \le 3  \}.
\end{equation}
Let $\mathrm{m}_v$ denote the measure on $\Omega_v$ which selects $\vartheta_i(v)$ with probability $p_i$. Moreover, for $v \in \Omega$ and $w \in \Omega_v$ let $B(v,w)$ be the face of the cube $\cQ_r$ such that a particle traveling with velocity $v$ colliding with that face would adopt the velocity $w$. Formally, for $v \in \Omega$ and $w = \vartheta_k(v)$ 

\begin{equation}
  B(v,w) = \{b \in \partial \cQ_r : b_k = - \frac{v_k}{\abs{v_k}}r \}.
\end{equation}

\subsection{Markovian Flight Process}
\label{ss:Markovian Flight Process}

Let $\{u_n \}_{n=0}^{\infty}$ be a realization of the following Markov chain on $\Omega$: $u_1 = v_0$ and then for all $i \ge 1$, $u_{i+1}$ are independently selected from $\Omega_{u_i}$ according to the measure $\mathrm{m}_{u_i}$. For later use let $u_0 \in \Omega_{v_0}$. Let 

\begin{equation} \label{exp}
  \{\xi_n \}_{n=1}^{\infty} \sim EXP(1)
\end{equation}
be i.i.d exponentially distributed \emph{flight times} and let 

\begin{equation} \label{disc MFP}
  Y_n := \sum_{i=1}^n y_i \quad , \quad y_n := \xi_n u_n
\end{equation}
denote the \emph{discrete Markovian Flight Process}. To define the continuous process, for $t\in \R$ let

\begin{equation}\label{tau and nu for Y}
  \tau_n := \sum_{i=1}^n \xi_i \quad , \quad \nu_t := \max\{ n : \tau_n \le t \} \quad , \quad \{t \} := t - \tau_{\nu_t},
\end{equation}
that is $\tau_n$ are the scattering times, $\nu_t$ is the label of the most recent scattering, and $\{t\}$ is the time since the previous scattering, at time $t$. Now define

\begin{equation}\label{MFP}
  Y(t) := Y_{\nu_t} + u_{\nu_t+1} \{t\}
\end{equation}
to be the \emph{(continuous) Markovian Flight Process}. Note that the processes $t \mapsto Y(t)$ and $\{Y_n\}_{n=1}^{\infty}$ do not depend on $r$.

For later use we introduce the following \emph{virtual scatterers}:

\begin{gather}
  Y^{\prime}_k := Y_k + \beta_k \quad , \quad \beta_k \sim UNI(-B(u_k,u_{k+1})) \quad , \quad k \ge 0 \\
  \cS^Y_n : = \{Y^{\prime}_k\in \R^3, \;\;\; 0 \le k \le n \}  \qquad, \qquad n \ge 0. \notag
\end{gather}
In words $Y^{\prime}_k$ is the position of a scatterer \emph{if it had caused} the $k^{th}$ collision (of course $Y$ is independent of any scatterers, thus the term virtual). Note also that we assume there is a virtual collision at time $0$, this has no effect on the definition of the model however will ease the notation. One difference with the random Lorentz gas is that the position of a scatterer associated to a velocity jump is not uniquely determined. Therefore we select from among the possible virtual scatterers uniformly. 

For later use we introduce the sequence of indicators $\epsilon_j = \one \{ \xi_j <1\}$ and the corresponding distributions $EXP(1|1):= \operatorname{distrib}(\xi | \epsilon=1)$ and similarly $EXP(1|0) = \operatorname{distrib}(\xi | \epsilon =0)$. We refer to $\underline{\epsilon}:= (\epsilon_j)_{j \ge 0}$ as the \emph{signature} of the sequence $(\xi_j)_{j\ge 0}$.

\subsection{Joint Construction}
\label{ss:Joint Construction}

Our goal for this section is to construct the physical wind-tree and Markovian processes on the same probability space. We construct the wind-tree process as an \emph{exploration process}: in that the process explores its environment as time moves forward. For convenience for what follows we will also construct a third \emph{auxiliary process}, $\{t \mapsto Z(t)\}$, coupled to the $X$ and $Y$ processes. The auxiliary process, which we call either the \emph{forgetful} or \emph{myopic} process, is only used in Sections \ref{s:Beyond the Naive Coupling} - \ref{s:Proof of Intra}.  Hence some readers may wish to ignore it until later. Indeed if we only wanted to prove Theorem \ref{thm:main theorem} for times of order $o(r^{-1})$ (we do this in Section \ref{s:No mismatches}) then this myopic process does not play a role and can be ignored.

The construction will proceed inductively on certain (as yet unspecified) time intervals. To simplify the explanation, first we will explain how the processes $X$ and $Z$ are constructed on a given time interval, given certain random data. Then, we will explain how the random data is generated to enable the coupling to $\{t \mapsto Y(t)\}$ and we will explain on which time intervals these processes are defined.

Throughout the construction we label the velocity of $\dot{X}(t) =: V(t)$, $\dot{Y}(t)=:U(t)$ and $\dot{Z}(t) = W(t)$.

\subsubsection{Building $X$ on $[\wh{\tau}_{n-1},\wh\tau_{n})$}
\label{sss:X construct}

We label the intervals of construction of $X$ by $[\wh\tau_{n-1}, \wh\tau_n)$. In Subsection \ref{sss:Joint Coupling} we will make precise what these $\wh\tau$ are.

To construct $X$ on an interval $[\wh\tau_{n-1},\wh\tau_{n})$, given a position $X(\wh\tau_{n-1}) = X_{n-1} \in \R^3$, a velocity $V(\wh \tau_{n-1}^+) \in \Omega$ and $\cS^X_{n-1} \subset \R^{n-1} \cup \{\bigstar\}$ a finite set of points (where $\bigstar$ is a fictitious point at infinity with $\inf_{x \in \R^3} \abs{x - \bigstar} = \infty$ which will aid in the exposition) perform the following steps:

\begin{enumerate} [{[}Step 1{]}]
  \item \textbf{Mechanical flight on $\cS^X_{n-1}$ in $[\wh\tau_{n-1},\wh\tau_{n})$:}
    \label{stepone - X}
    The trajectory $ t \mapsto X(t)$ on $t \in [\wh\tau_{n-1},\wh\tau_{n})$ is defined to be free motion, with initial position $X_{n-1}$ and velocity $V(\wh \tau_{n-1}^+)$, and with reflective collisions on $\cQ_r + \cS^X_{n-1}$.
  \item \textbf{Attempt Fresh Collision:}
    \label{steptwo - X}
    Suppose, we are given a velocity $\wh w_{n+1} \in \Omega_{V(\wh\tau_n^-)}$ and an impact parameter $\wh{\beta}_n \in - B(V(\wh\tau_n^-),\wh w_{n+1})$. Set
    \begin{align}
      X^{\prime \prime} := X(\wh\tau_{n}) + \wh{\beta}_n
    \end{align}
    Now

    \begin{itemize}
      \item If $\exists 0 <s \le \wh\tau_{n-1} : X(s) \in X^{\prime\prime}+\cQ_r $ then let $X_n^{\prime}:=\bigstar$, and $V(\wh\tau_n^+)=V(\wh\tau_n^-) $.
      \item If not, then $X_n^{\prime}:=X^{\prime\prime}$, and $V(\wh\tau_n^+)=\wh w_{n+1}$.
    \end{itemize}
    
    Now set $\cS^X_n = \cS^X_{n-1} \cup \{X_n^{\prime}\}$. 
\end{enumerate}

\noindent We say: on the interval $[\wh\tau_{n-1},\wh\tau_n)$ the process $\{t \mapsto X(t)\}$ \emph{attempts} a \emph{fresh collision} at $\wh\tau_n$  with \emph{data} $(\wh w_{n+1},\wh{\beta}_n)$.

We will make precise the distributions of $\wh w_{n+1}$ and $\wh{\beta}_n$ in the construction below. Note that if, given a $\wh w_{n+1}$ and a $\wh{\beta}_n$, we build $X$ on the interval $[\wh\tau_{n-1},\wh\tau_n)$ then, after the construction we have sufficient information to build $X$ on the interval $[\wh\tau_{n},\wh\tau_{n+1})$ (provided we are given another pair $\wh w_{n+2},\wh{\beta}_{n+1}$).

\subsubsection{Building $Z$ on $[\wt\tau_{n-1},\wt\tau_n)$}
\label{sss:Z construct}

We call the process $\{t \mapsto Z(t)\}$ forgetful in that the process only respects \emph{direct mismatches} (see Figure \ref{fig:direct} for a diagram). That is, recollisions with the immediately preceding scatterer, or shadowed events where the scattering is shadowed by the immediately preceding path segment.

\begin{figure}[ht!]
  \begin{center}    
    \includegraphics[width=0.9\textwidth]{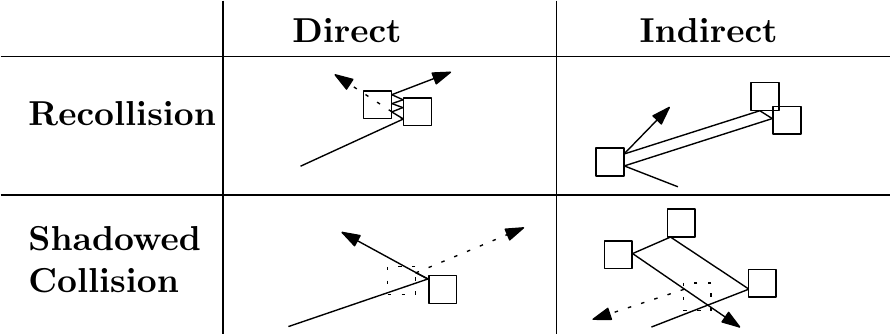}
  \end{center}
  \caption{%
    {\tt In the above diagram we show examples of direct and indirect, recollisions and shadowed events. In each case the path of the Markovian process is in dotted line while the wind-tree process is in solid line. Additionally, virtual scatterers are in dotted line while actual scatterers for the $X$ process are in solid line.}   
  }%
  
  \label{fig:direct}

\end{figure}

Suppose that we are given a time interval $[\wt\tau_{n-1},\wt\tau_n)$. Assume further, we are given a position $Z(\wt\tau_{n-1}) = Z_{n-1}$, velocity $W(\wt\tau_{n-1}^+) \in \Omega$, and a pair $\cS^Z_{n-1} = \{Z_{n-1}^\prime,Z_{n-2}^\prime\} \subset \R^3 \cup \{ \bigstar\}$.

\begin{enumerate} [{[}Step 1{]}]
\item \textbf{Mechanical flight on $\cS^Z_{n-1}$ in $[\wt\tau_{n-1},\wt\tau_n)$:}
    \label{stepone - Z}
    The trajectory $ t \mapsto Z(t)$ on $t \in [\wt\tau_{n-1},\wt\tau_n)$ is defined to be free motion starting at position $Z(\wt\tau_{n-1})$ and with velocity $W(\wt\tau_{n-1}^+)$ with reflective collisions on $\cQ_r + \cS^Z_{n-1}$.
  \item \textbf{Attempt Fresh Collision:}
    \label{steptwo - Z}
    Suppose that we are given a velocity $\wt w_{n+1} \in \Omega_{W(\wt\tau_n^-)}$ and an impact parameter $\wt{\beta}_n \in - B(W(\wt\tau_n^-),\wt w_{n+1})$. Set
    \begin{align}
      Z^{\prime \prime} := Z(\wt\tau_n) + \wt{\beta}_n
    \end{align}
    Now

    \begin{itemize}
      \item If there exists an $s \in  ( \wt\tau_{n-2} , \wt\tau_{n-1}] : Z(s) \in Z^{\prime\prime}+\cQ_r $ then let $Z_{n}^{\prime}:=\bigstar$, and $W(\wt\tau_n^+)=W(\wt\tau_n^-) $.
      \item If not, then $Z_{n}^{\prime}:=Z^{\prime\prime}$, and $Z(\wt\tau_n^+)=\wt w_{n+1}$.
    \end{itemize}
    
    Now set $\cS^Z_{n} = \{Z^{\prime}_{n} , Z_{n-1}^{\prime}\}$. 
\end{enumerate}

Similarly we say that on the interval $[\wt\tau_{n-1},\wt\tau_n)$ the process $\{t \mapsto Z(t)\}$ \emph{attempts} a \emph{fresh collision} at $\wt\tau_n$  with \emph{data} $(\wt w_{n+1},\wt{\beta}_n)$.

\subsubsection{Parity}

Consider just the processes $\{t \mapsto Y(t)\}$ and $\{t \mapsto X(t)\}$, the idea behind the coupling is the following:

\begin{itemize}[$\circ$]
  \item $X(0) = Y(0)$ and the velocities are initially parallel.
  \item $X$ and $Y$ then run parallel until one of two possible \emph{mismatches} occurs:
    \begin{itemize}[$\circ\circ$]
    \item A \emph{recollision}, which corresponds to a collision with a previously placed scatterer during \ref{stepone - X} of Subsection \ref{sss:X construct}.
    \item A \emph{shadowed collision}, which corresponds to $X_n^\prime = \bigstar$ in \ref{steptwo - X} of Subsection \ref{sss:X construct}.
    \end{itemize}
  \item After a mismatch the two velocity processes proceed independently.
  \item When the two velocities happen to coincide we recouple the two processes and they run parallel until the next mismatch.
\end{itemize}

However there is a problem with this setup as we have described it. Note that there are two \emph{parity classes}:  $(\vect{v}, (\vartheta_i(\vartheta_j(\vect{v})))_{i \neq j})$ and $(-\vect{v}, (\vartheta_i(\vect{v}))_{i =1,2,3})$. The Markov process $(u_n)_{n\in \N}$ alternates between these two classes. The problem is that if there is a parity mismatch between $V(t)$ and $U(t)$ at a given time, then as long as the two processes experience fresh collisions at the same times, only another mismatch can restore the parity. This is too long to wait. Therefore we need to alter the sequence of collision times to restore parity. For this we will make use of Lemma \ref{lem:parity}. For future use, we define the equivalence relation $u \parity v$ if $u$ and $v$ are in the same parity class.

\begin{lemma} \label{lem:parity}
Let $(\tau_j)_{j\ge1}$ be the points of a Poisson point process of intensity 1 on $\R_+$. Form a new sequence as follows: sample $\xi^\prime \sim EXP(1)$, independently of the sequence  $(\tau_j)_{j\ge1}$. Let the new sequence $(\tau^\prime_j)_{j\ge1}$ be as follows: 

\begin{itemize}
  \item
   If $\xi^\prime<\tau_1$ then $\tau^\prime_1=\xi^\prime$, and $\tau^\prime_j=\tau_{j-1}$ for $j\ge2$.(That is: insert $\xi^\prime<\tau_1$ as the first point and leave the rest as they are.)
  \item
   If $\xi^\prime>\tau_1$ then $\tau^\prime_j=\tau_{j+1}$ for $j\ge1$. (That is: delete the first point $\tau_1$ and leave the rest as they are.)
\end{itemize}
\end{lemma}

\begin{proof}
  Consider the distribution of $\tau^{\prime}_1$

  \begin{align*}
    \probab{\tau^{\prime}_1> t} &= \probab{\xi > t, \;\; \xi < \tau_1} + \probab{\tau_2 >t,\;\;  \xi > \tau_1}\\
    &= e^{-t} \condprobab{\tau_1 >\xi}{\xi >t} + \probab{\xi>\tau_1}\condprobab{\tau_2>t}{\xi>\tau_1}
  \end{align*}
  where we have used the definition of conditional probability and the fact that $\xi$ is exponentially distribution. Now note that $\condprobab{\tau_2>t}{\xi>\tau_1} = \condprobab{\xi>t}{\xi>\tau_1}$ since $\tau_2$ and $\xi$ are both exponentially distributed conditioned to be larger that $\tau_1$. Therefore

  \begin{align*}
    \probab{\tau^{\prime}_1> t} &= e^{-t} \condprobab{\tau_1 >\xi}{\xi >t} + \probab{\xi>\tau_1}\condprobab{\xi>t}{\xi>\tau_1}\\
                              &= e^{-t} \condprobab{\tau_1 >\xi}{\xi >t} + e^{-t}\condprobab{\xi>\tau_1}{\xi>t}\\
                              &= e^{-t} \condprobab{\tau_1 >\xi}{\xi >t} + e^{-t}(1-\condprobab{\tau_1>\xi}{\xi>t})\\
                              &= e^{-t}.\\
  \end{align*}

  Turning now to the distribution $\tau_2^{\prime} -\tau_1^{\prime}$ (all the other increments are clearly i.i.d exponentially distributed)

  \begin{align*}
    \probab{\tau_2^{\prime} - \tau_{1}^{\prime} > t } & = \probab{\tau_1 - \xi > t , \tau_1 > \xi} + \probab{\tau_{3} -\tau_2 > t , \tau_1 < \xi}\\
    & = e^{-t} \probab{\tau_1 > \xi} + e^{-t} \probab{\tau_1 < \xi} = e^{-t}.
  \end{align*}

  Finally, we look at the joint distribution

  \begin{align*}
    \probab{\tau_1^{\prime}> t, \tau_2^{\prime} - \tau^{\prime}_1>s} = \probab{\xi> t, \tau_1-\xi >s} + \probab{\xi > \tau_1, \tau_2>t, \tau_3 -\tau_2>s}.
  \end{align*}
  By construction $\tau_3 - \tau_2$ is exponentially distributed and independent of $\tau_1,\tau_2,\xi$, thus

  \begin{align*}
    \probab{\tau_1^{\prime}> t, \tau_2^{\prime} - \tau^{\prime}_1>s} &= \probab{\xi> t, \tau_1-\xi >s} + \probab{\xi > \tau_1, \tau_2>t} e^{-s}\\
    &= \condprobab{\xi> t, \tau_1-\xi >s}{\xi< \tau_1} \probab{\xi< \tau_1} + \probab{\xi > \tau_1, \tau_2>t} e^{-s}.
  \end{align*}
  Conditioned on $\xi<\tau_1$, $\tau_1-\xi$ is exponentially distributed independently of $\xi$. Thus

  \begin{align*}
    \probab{\tau_1^{\prime}> t, \tau_2^{\prime} - \tau^{\prime}_1>s}
    &= e^{-s}\condprobab{\xi> t}{\xi< \tau_1} \probab{\xi< \tau_1} + \probab{\xi > \tau_1, \tau_2>t} e^{-s}\\
    &= e^{-s}\probab{\xi>t, \xi< \tau_1} + \probab{\xi > \tau_1, \tau_2>t} e^{-s}\\
    &= e^{-s} \probab{\tau_1^\prime > t}\\
    &= e^{-s}e^{-t}.
  \end{align*}

\end{proof}

\subsubsection{Joint Coupling}
\label{sss:Joint Coupling}

Assume $\{t \mapsto Y(t)\}$ is constructed as in Subsection \ref{ss:Markovian Flight Process}. We will construct the $X$ and $Z$ processes inductively on the intervals $[\tau_{2n}, \tau_{2n+2})$ as follows: First set

\begin{align}
  \begin{aligned}
    &X(0)=X_0= 0 \quad , \quad V(0^+) = u_1 \quad , \quad X_0^\prime = \wh\beta_0=\beta_0 \quad , \quad \cS_0^X = \{X_0^\prime\}                        \\
    &Z(0)=Z_0=0 \quad , \quad W(0^+) = u_1 \quad , \quad W_0^\prime = \wt\beta_0=\beta_0 \quad , \quad \cS_0^Z = \{Z_0^\prime,Z_{-1}^\prime \}
  \end{aligned}
\end{align}
where $Z_{-1}^\prime = \bigstar$. Let $n \in \N$ and sample an exponential time $\zeta_n \sim EXP(1)$ independent of the entire history up to this point. In which case there are $7$ possible situations arranged and labelled in the following table:

\begin{center}
  \begin{tabular}{ |p{4cm}|p{4cm}|p{4cm}| }
    \hline
    Parity at time $\tau_{2n}^+$ &  $\zeta_n \le \xi_{2n+1} $  & $ \zeta_n > \xi_{2n+1}$\\
    \hline
    $U\parity V \parity W$ &  \multicolumn{2}{|c|}{A}  \\
    \hline
    $U\not\parity V \parity W$ & B & C \\
    \hline
    $U\parity V \not\parity W$ & D & E \\ 
    \hline
    $U\parity W \not\parity V$ & F & G \\ 
    \hline
  \end{tabular}
\end{center}
For completeness of the construction we define all of these cases, however on our time scales we will (w.h.p) only see situations A, B, and C.

On the interval $[\tau_{2n}, \tau_{2n+2})$ the $X$ and $Z$ processes attempt fresh collisions at the following times:

\begin{center}
  \begin{tabular}{ |p{4cm}|p{4cm}|p{4cm}| }
    \hline
    Situation &  $X$  &  $Z$\\
    \hline
    A & $\tau_{2n+1}$, $ \tau_{2n+2}$ & $\tau_{2n+1}$, $ \tau_{2n+2}$ \\
    \hline
    B & $\tau_{2n}+\zeta_n$, $\tau_{2n+1}$, $\tau_{2n+2}$ & $\tau_{2n}+\zeta_n$, $\tau_{2n+1}$, $\tau_{2n+2}$ \\
    \hline
    C & $\tau_{2n+2}$ & $ \tau_{2n+2}$ \\
    \hline
    D & $\tau_{2n+1}$, $ \tau_{2n+2}$ & $\tau_{2n}+\zeta_n$, $\tau_{2n+1}$, $\tau_{2n+2}$ \\
    \hline
    E & $\tau_{2n+1}$, $ \tau_{2n+2}$ & $\tau_{2n+2}$ \\
        \hline
    F &  $\tau_{2n}+\zeta_n$, $\tau_{2n+1}$, $\tau_{2n+2}$ & $\tau_{2n+1}$, $ \tau_{2n+2}$  \\
    \hline
    G &  $\tau_{2n+2}$ & $\tau_{2n+1}$, $ \tau_{2n+2}$  \\
    \hline
  \end{tabular}
\end{center}

In what follows the following \textbf{coupling rule} will dictate the random variables $\wh\beta_n, \wh w_n, \wt\beta_n, \wt w_n$ used in the attempted fresh collisions.

\noindent\textbf{\underline{For the $Z$-process}:} If the $Z$-process is to attempt a fresh collision at time $t_a$, sample $\wt{w}$ from $\Omega_{W(t_a^-)}$ according to the measure $\mathrm{m}_{W(t_a^-)}$ and sample $\wt{\beta}$ from $-B(W(t_a^-),\wt{w})$ both independent of the past. We now attempt to couple $W$ with $U$ at $t_a$:

\begin{itemize}
  \item \textbf{\underline{Couple $W$ to $U$:}} If $W(t_a^-) = U(t_a^-) $ and $t_a = \tau_{n}$ for some $n$, attempt a fresh collision at $Z(t_a)$ using data  $(\beta_n,u_{n+1})$.
  \item \textbf{\underline{$W$ is independent of $U$:}} Otherwise attempt a fresh collision at $Z(t_a)$ using data $(\wt\beta,\wt w)$.
\end{itemize}

\noindent\textbf{\underline{For the $X$-process}:} If the $X$-process is to attempt a fresh collision at time $t_a$, sample $\wh{w}$ from $\Omega_{V(t_a^-)}$ according to the measure $\mathrm{m}_{V(t_a^-)}$ and sample $\wh{\beta}$ from $-B(V(t_a^-),\wh{w})$ both independent of the past. We now couple $V$ to either $U$ \emph{and/or} $W$ if possible: 

\begin{itemize}
  \item \textbf{\underline{Couple $V$ to $U$:}} If $V(t_a^-)   = U(t_a^-)$ and $t_a = \tau_{n}$ for some $n$ attempt a fresh collision at $X(t_a)$ using $(\beta_n,u_{n+1})$.
  \item \textbf{\underline{Couple $V$ to $W$:}} If $V(t_a^-) = W(t_a^-)$ and the $Z$ process also attempts a fresh collision \emph{independent of $U$} at time $t_a$, attempt a fresh collision at $X(t_a)$ using  $(\wt\beta,\wt w)$.
  \item \textbf{\underline{$V$ is independent of $U$ and $W$:}} Otherwise attempt a fresh collision at $X(t_a)$ using $(\wh\beta,\wh w)$.

\end{itemize}

After this construction we have generated two processes. For the wind-tree exploration process $\{t \mapsto X(t)\}$, the \emph{attempted fresh collision} times are $\{ \wh\tau_n\}_{n \in \N}$, by Lemma \ref{lem:parity} these form a (temporal) Poisson point process on $\R_+$; the scatterers are placed at positions $\{X_n^\prime\} \subset \R^3 \cup \{\bigstar\}$; and the impact parameters are $\{\wh\beta_n\}_{n \in \N}$. Moreover, the \emph{attempted} velocities after collisions are $\{\wh w_n\}_{n \in \N}$, these velocities are attempted since, in \ref{steptwo - X} the attempted collision may be rejected (i.e $X_n^\prime = \bigstar$). Because of the Poisson distribution of the scatterers in $\R^3$ this process is distributed like the original wind-tree model as described in the introduction.

For the process $\{t \mapsto Z(t)\}$, the \emph{attempted fresh collision} times are $\{ \wt\tau_n\}_{n \in \N}$, which by Lemma \ref{lem:parity} form a (temporal) Poisson point process on $\R_+$; the scatterers are placed at positions $\{Z_n^\prime\} \subset \R^3 \cup \{\bigstar\}$; and the impact parameters are $\{\wt\beta_n\}_{n \in \N}$. The \emph{attempted} velocities for the $Z$-process are $\{\wt w_n \}_{n \in \N}$.

\subsection{Main Technical Result and Method Proof}
\label{ss:Main Technical Result}

The main result we prove is the following


\begin{theorem} \label{thm:main-technical}
  Let $T=T(r)$ be such that $\lim_{r \to 0} T(r)=\infty$ and $\lim_{r\to 0} r^2 T(r) = 0$. Then for any $\delta>0$
  \begin{equation}\label{X to Y}
    \lim_{r\to 0} \probab{\sup_{0\le t \le T}\abs{X(t) -Y(t)}>\delta \sqrt{T}} = 0.
  \end{equation}
\end{theorem}
From here Theorem \ref{thm:main theorem} follows as a consequence of the classical Donsker's invariance principle \cite{billingsley_68}: that is, the process $t \mapsto Y(t)$ is a true Markov process, hence Donsker's original invariance principle does not apply directly, however in what follows we will show how to separate $Y$ into i.i.d mean $0$ pieces with finite second moment. Thus Donsker's principle will imply that $t \mapsto \frac{Y(tT)}{\sqrt{T}}$ converges to a Wiener process in the diffusive scaling. Therefore the process $t\mapsto X(t)$ does as well. We omit the details of this final step and the rest of the paper is devoted to proving Theorem \ref{thm:main-technical}.

The strategy of proof is the same as in \cite{lutsko-toth_19}. We begin with the joint realization of the Markovian flight process and the wind-tree exploration process described above. During the two mismatch events (recollisions and shadowed scatterings) the two velocity processes diverge. In either case the two processes are decoupled until recoupling is possible. At which point the two processes are recoupled and proceed parallel to each other until the next mismatch.

The proof then follows two steps. In Section \ref{s:No mismatches} we show that such mismatches occur only on time scales of order $r^{-1}$. Hence until such times both process are (w.h.p) in the the same position and Theorem \ref{thm:main-technical} follows immediately for $T=o(r^{-1})$. Note that this intermediate result is a statement about the Markovian flight process. During the rest of the paper we show that on time scales of order $o(r^{-2})$ only (geometrically) simple mismatches occur. During such mismatches the separation between $X$ and $Y$ is of order $\cO(1)$. Hence on the time scales of Theorem \ref{thm:main-technical} there are $o(Tr)$ mismatches. During each mismatch the two processes separate by a distance of order $\cO(1)$, hence up to $T = o(r^{-2})$,  $\frac{\abs{X(T(r))-Y(T(r))}}{\sqrt{T}} \to 0$, thus proving \eqref{X to Y}. Sections \ref{s:Beyond the Naive Coupling}-\ref{s:Proof of Intra} are devoted to formalizing this argument.

The reason for introducing the forgetful process $\{t \mapsto Z(t)\}$ is that the forgetful process will satisfy additional independence properties exploited in the proof. Thus during the second stage of the prove, we will in fact show that the forgetful and Markovian processes do not diverge too much. Then we show that with high probability the wind-tree and forgetful processes are in fact the same on these time scales (i.e we show that with probability tending to $1$ as $r \to 0$, the direct mismatches defining the $Z$-process are the only ones seen by the $X$-process).

\noindent \emph{Remark on dimension:} As with the Lorentz gas, because of the recurrence of the random walk the same proof does not yield the result in 2 dimensions. For the Lorentz gas the geometry of mismatches imposed another reason that the proof cannot be extended to 2 dimensions. However for the wind-tree model the mismatches have a far simpler geometry and thus this obstruction is not present in 2 dimensions. 

\subsection{$r$-consistency and $r$-compatibility}
\label{ss:Inconsistency}

The proof will hinge on two definitions which we present now for a general process (i.e this could be a segment of any of the above mentioned processes). Let
\begin{align*}
  n\in\N,  \qquad
  \tau_0\in\R, \qquad
  \cZ_0\in\R^3,  \qquad
  U_0,\dots, U_{n+1} \in \Omega  \qquad
  t_1,\dots,t_n\in\R_+,
\end{align*}
be given, such that either $U_{i+1} \in \Omega_{U_i}$ or $U_{i+1}=U_i$ for all $0\le i \le n$. Moreover fix a set of vectors $\beta_j \in B(U_{j},U_{j+1})$ (if $U_j = U_{j+1}$ we set $\beta_j = \bigstar$) and define for $j=0, \dots, n$, 
\begin{align*}
  \tau_j:= \tau_0+\sum_{k=1}^j t_k,  \qquad
  \cZ_j:= \cZ_0+\sum_{k=1}^j t_kU_k,     \qquad
  \cZ_j^{\prime} :=  \cZ_j+ \beta_j
\end{align*}
and for $t\in[\tau_j, \tau_{j+1}]$, $j=0, \dots, n$,  
\begin{align*}
  \cZ(t) := \cZ_j+ (t-\tau_j)U_{j+1}. 
\end{align*}
We call the piece-wise linear trajectory $\big(\cZ(t): \tau_0^- < t < \tau_n^+\big)$ mechanically \emph{$r$-consistent} if 
\begin{align}
  \label{consist}
  \not\exists\; t \in [\tau_0,\tau_n],\; j \in \{0,\dots,n\} : \cZ(t)-\cZ_j^{\prime} \in \cQ_r^o
\end{align}
($\cQ_r^o$ denotes the interior) and \emph{$r$-inconsistent} if \eqref{consist} fails.

Given two finite pieces of mechanically $r$-consistent trajectories $\big(\cZ_{a}(t): \tau_{a,0}^- < t < \tau_{a,n_a}^+\big)$ and $\big(\cZ_{b}(t): \tau_{b,0}^- < t < \tau_{b,n_b}^+\big)$, defined over non-overlapping time intervals: $[\tau_{a,0},\tau_{a,n_a}] \cap [\tau_{b,0},\tau_{b,n_b}]=\emptyset$ with $\tau_{a,n_a}\le\tau_{b,0}$,  we will call them mechanically \emph{$r$-compatible} if
\begin{align}
  \label{compat}
  \begin{aligned}
    \not\exists \; t \in [\tau_{a,0},\tau_{a,n_{a}}], j \in \{0,\dots n_b\} : \cZ_a(t) - \cZ_{b,j}^\prime \in \cQ_r^o ,\\  \mbox{ and } \not\exists \; t \in [\tau_{b,0},\tau_{b,n_{b}}], j \in \{0,\dots n_a\} : \cZ_b(t) - \cZ_{a,j}^\prime \in \cQ_r^o 
\end{aligned}
\end{align}
mechanical trajectories are $r$-\emph{incompatible} if \eqref{compat} fails.

\section{No Mismatches Till $T=o(r^{-1})$}
\label{s:No mismatches}

\subsection{Excursions}
\label{ss:Excursions}

Unlike in the 3-dimensional Lorentz gas case the directions of path segments of the Markovian flight process are not independent. To decompose the process $t \mapsto Y(t)$ into i.i.d segments we introduce \emph{excursions}. Let

\begin{equation}\label{gamma def}
  \gamma : = \min \{ i>1 : u_{i+1} = v_0 \}
\end{equation}
and define a \emph{pack} to be a collection

\begin{equation*}
  \varpi := \left( \gamma ; \{u_i\}_{i=1}^{\gamma}, \{\beta_i\}_{i=1}^{\gamma}  , \{ \xi_i \}_{i=1}^{\gamma}\right),
\end{equation*}
$u_{\gamma} \in \Omega_{v_0}$, and for all $i >1$, $u_{i} \neq v_0$ and $u_{i-1} \in \Omega_{u_i}$. Given a pack we consider the process $t \mapsto Y(t)$ associated to it via the rules set forth in Section \ref{ss:Markovian Flight Process} - call the process built from such a pack, \emph{an excursion}.

\subsection{Concatenation}
\label{ss:Concatenation}

For $n=1,2,3,\dots$ consider infinitely many independent packs:
\begin{equation*}
  \varpi_n = \left( \gamma_n, \{u_{n,i}\}_{i=1}^{\gamma_n},\{\beta_{n,i}\}_{i=1}^{\gamma_n} , \{\xi_{n,i}\}_{i=1}^{\gamma_n}\right).
\end{equation*}
For each pack define the associated flight process $t \mapsto Y_n(t)$ together with the discrete process $\{Y_{n,i}\}_{i=0}^{\gamma_n}$. Denote

\begin{equation*}
  \theta_n := \sum_{i=1}^{\gamma_n} \xi_{n,i}, \qquad  \qquad \overline{Y_n} := Y_{n,\gamma_n}.
\end{equation*}
Define the following variables

\begin{gather*}
  \begin{gathered}
    \Gamma_0 = 0, \qquad  \qquad \Gamma_n = \Gamma_{n-1} + \gamma_n, \quad \mbox{ for } n \ge 1\\
    \nu_n := \max \{m: \Gamma_n \le n \}, \qquad \qquad \{n \} := n - \Gamma_{\nu_n}.
  \end{gathered}
\end{gather*}
Likewise

\begin{gather*}
  \begin{gathered}
    \Theta_0 = 0, \qquad  \qquad \Theta_n = \Theta_{n-1} + \theta_n, \quad \mbox{ for } n \ge 1\\
    \nu_t := \max \{m: \Theta_n \le t \}, \qquad \qquad \{t \} := t - \Theta_{\nu_t}.
    \end{gathered}
\end{gather*}

Now define the following three processes: the \emph{end-point process} with $\Xi_0=0$

\begin{equation*}
  \Xi_n := \sum_{k=1}^n \overline{Y_k},
\end{equation*}
the \emph{concatenated discrete Markovian flight process} with $Y_0=0$

\begin{equation*}
  Y_n := \Xi_{\nu_n} + Y_{\nu_n+1,\{n\}},
\end{equation*}
and the continuous \emph{concatenated Markovian flight process} with $Y(0)=0$

\begin{equation*}
  Y(t):= \Xi_{\nu_t} + Y_{\nu_t+1}(\{t\}).
\end{equation*}
The advantage of this decomposition is that the different excursions making up the process $Y$ are i.i.d steps with exponentially decaying tails.


\subsection{Occupation Measures}
\label{ss:Occumation Measures}

Define the following occupation measures for a set $A \subset \R^3$

\begin{align}\notag
  & G(A) := \expect{ \abs{\{ 1 \le k < \infty : Y_k \in A \}}} ,
  & H(A) := \expect{ \abs{\{ 0 < t <\infty : Y(t) \in A \}}},\\
  \notag
  & g(A) := \expect{ \abs{\{ 1 \le k \le \gamma_{1} : Y_k \in A \}}},
  & h(A) := \expect{ \abs{\{ 0 < t < \Theta_1 : Y(t) \in A\}}},\\
  \notag
  & R(A) := \expect{ \abs{\{ 1 \le k < \infty : \Xi_k \in A \}}}.
\end{align}


\begin{lemma} \label{lem:greenbounds}

  The following upper bounds hold for any measurable set $A \subset \R^3$

  \begin{align}
    & R(A) \le K(A) + L_{v_0}(A) \label{R bound},\\
    & g(A) \le M(A) + L_{v_0}(A), \qquad  \qquad h(A) \le M(A) + L_{v_0}(A)  \label{g bound},\\
    & G(A) \le K(A) +L_{v_0}(A),   \qquad  \qquad H(A) \le K(A) + L_{v_0}(A)  \label{G bound},
  \end{align}

  where

  \begin{gather*}
    K(dx) :=C\min\{1,\abs{x}^{-1}\}dx  \quad, \quad M(dx) :=  Ce^{-c\abs{x}}dx\\
    L_{v_0}(A) : = C\int_0^{\infty}\one \{tv_0 \in A\} e^{-ct}dt 
  \end{gather*}

\end{lemma}

This Lemma is slightly different from the Lorentz gas case as $L_{v_0}$ takes into account the discrete state-space of velocities. However the end result (Proposition \ref{prop:inter-excursion}) remains the same.

\begin{proof}

  To bound $g(A)$ let

  \begin{equation}\notag
    g_1(A) := \probab{Y_1 \in A} = C \int_0^{\infty} \one\{tv_0 \in A\} e^{-t}dt.
  \end{equation}
  We have fixed the initial velocity to be $u_1 = v_0$, therefore the points $\{Y_k - Y_1 \}_{k=1}^{\gamma_1}$ are independent of the initial step $Y_1$. Therefore write

  \begin{equation}
     g_2(A) := \expect{\abs{\{ 1\le k \le \gamma_{1} : Y_k - Y_1 \in A\}}}, \notag
  \end{equation}
  and note that

  \begin{equation} \label{g conv}
    g(A) = \int_{\R^3} g_2(A-x) g_1(dx).
  \end{equation}
  Similarly we can write

  \begin{align}
    h_1(A) & := \expect{\abs{\{t\le \tau_1 : Y(t) \in A\}}} = C \int_0^{\infty} \one\{tv_0\in A\} e^{-\max\{1,t\}} dt, \notag \\
    h_2(A) & := \expect{\abs{\{ \tau_1 \le t \le \Theta_1 : Y(t) -Y_1 \in A\}}} \notag\\
    \label{h conv}
    h(A)    &=    \int_{\R^3}h_2(A-x) g_1(dx) +h_1(A).
  \end{align}

  Now the bounds \eqref{g bound} follow by inserting the bounds:

  \begin{align}
    \begin{aligned}
      \label{g2 bounds}
    & g_2(\{x : \abs{x} >s \}) \le C e^{-cs},  &&& h_2\left(\{x:\abs{x}>s\}\right) \le Ce^{-cs}\\
    & g_2(\R^3) = \expect{\gamma_{1}}< \infty,  &&& h_2(\R^3) = \expect{\Theta_1 -\tau_1 } < \infty 
    \end{aligned}
  \end{align}
  into \eqref{g conv} and \eqref{h conv}. That is,

  \begin{align}
    g(A) \le \int_{A^c} g_2(\{y:\abs{y} > \abs{x}\}) dx + C \int_A g_1(dx) \le M(A) + L_{v_0}(A)
  \end{align}
  and likewise for $h(A)$.

  Now, to achieve \eqref{R bound} note that since $\gamma_1 > 1$

  \begin{align}
    \probab{\Xi_1 \in A} \le \expect{\abs{ \{2 \le k \le \gamma_1 : Y_k \in A \}}} \le g(A) 
  \end{align}
  Hence the density of distribution of $\Xi_1$ is bounded by the density of $g$. Moreover, because $\probab{\theta_1 > s } \le C e^{-cs}$ for some $C< \infty$ and $c>0$, we know that the density of distribution of $\Xi_1$ has exponentially decaying tails. Therefore $\Xi$ is a random walk, with i.i.d steps, and step distribution bounded by $g$ with exponentially decaying tails. Hence a standard random walk argument implies \eqref{R bound}.





  \eqref{G bound} then follows by writing (using the fact that the different excursions are i.i.d)

  \begin{align}\notag
    & G(A) =  g(A) + \int_{\R^3} g(A-x) R(dx),
    & H(A) =  h(A) +  \int_{\R^3}h(A-x) R(dx)
  \end{align}
  and inserting \eqref{R bound} and \eqref{g bound}.  
\end{proof}

\subsection{Inter-Excursion Mismatches}
\label{ss:Mismatches Naive}

Let $t \to Y^{\ast}(t)$ denote a Markovian flight process with associated virtual scatterers $Y^{\ast\prime} \in \cS^{Y^{\ast}}$ and initial velocity $u_1^{\ast} \in -\Omega_{v_0}$. Let $t \to Y(t)$ be a second Markovian flight process with associated virtual scatterers $\cS^{Y}$, and initial velocity $v_0$.

 We think of $Y^{\ast}$ as the process run backwards in time. Define the events

\begin{align*}
  \wh W_j &:=  \big\{ \{ Y(t) - Y^{\prime}_k :  && 0 < t < \Theta_{j-1}, && \Gamma_{j-1}< k \le \Gamma_j\} \cap \cQ_r \neq \emptyset \big\}, \\
  \wt W_j &:=  \big\{  \{ Y^{\prime}_k-Y(t):  && 0\le k< \Gamma_{j-1}, && \Theta_{j-1}< t < \Theta_j \} \cap \cQ_r \neq \emptyset \big\},\\
  \wh W_j^* &:=  \big\{ \{ Y^*(t) - Y^{\prime}_k : && 0 < t < \Theta_{j-1}, && 0 < k \le \gamma\} \cap \cQ_r \neq \emptyset \big\}, \\
  \wt W_j^* &:=  \big\{  \{ Y^{*\prime}_k-Y(t): && 0 < k \le \Gamma_{j-1},&& 0 < t < \theta \} \cap \cQ_r \neq \emptyset \big\}, \\
  \wh W_\infty^* &:=  \big\{ \{ Y^*(t) - Y^{\prime}_k : && 0 < t < \infty, && 0 < k \le \gamma\} \cap \cQ_r \neq \emptyset \big\}, \\
  \wt W_\infty^* &:=  \big\{  \{ Y^{*\prime}_k-Y(t):  && 0 < k < \infty, && 0 < t < \theta \} \cap \cQ_r \neq \emptyset \big\}.
\end{align*}
In words $\wh W_j$ is the event that during the $(j-1)^{th}$ excursion, a collision of $Y$ is (virtually) \emph{shadowed} by a previous excursion. And $\wt W_j$ is the event that during the $(j-1)^{th}$ excursion the process (virtually) \emph{recollides} with a scatterer from an earlier excursion.

It readily follows that 
\begin{align}
  \label{hat-tilde-Ws}
  \begin{aligned}
    & \probab{\wh W_j} = \probab{\wh W_j^*} \le  \probab{\wh W_{j+1}^*} \le \probab{\wh W_{\infty}^*}, 
    \\
    & \probab{\wt W_j} =\probab{\wt W_j^*}\le  \probab{\wt W_{j+1}^*} \le \probab{\wt W_{\infty}^*}. 
  \end{aligned}
\end{align}
By the union bound

\begin{align}
  \label{pwhtbound-for-Y}
  \begin{aligned}
    \probab{\wh W^*_\infty} & \le \sum_{z\in \Z^{3}} \probab{ \{1<k<\infty: Y^*_k\in B_{zr,2r}\} \neq \emptyset } \probab{ \{0< t\le \theta: Y(t)\in B_{zr,2r}\} \not=\emptyset}\\
   & \le \sum_{z\in \Z^{3}} (2r)^{-1} \expect{ \abs{ \{1<k<\infty: Y^*_k\in  B_{zr,2r}\} } } \cdot\expect{ \abs{ \{0< t\le \theta: Y(t)\in  B_{zr,3r}\} } }
   \\
   \probab{\wt W^*_\infty} & \le  \sum_{z\in \Z^3}  \probab{ \{0<t<\infty: Y^*(t)\in B_{zr,2r}\} \neq \emptyset}\probab{ \{1\le  j \le \gamma : Y_j\in B_{zr,2r}\} \neq \emptyset}\\
   & \le \sum_{z\in \Z^{3}} (2r)^{-1} \expect{ \abs{ \{0<t<\infty: Y^*(t)\in B_{zr,3r}\} } }  \cdot \expect{\abs{\{1\le j\le \gamma :  Y_j\in B_{zr,2r}\}}}
\end{aligned}
\end{align}

\subsection{Computations}
\label{ss:Computations Naive}

\eqref{pwhtbound-for-Y} implies that

\begin{align}
  \begin{aligned}
    \label{W by Hg}
    & \probab{\wt{W}^{\ast}_{\infty}} \le (2r)^{-1} \sum_{z \in \Z^3 }H^{\ast}(B_{zr,3r})g(B_{zr,2r}) \\
    & \probab{\wh{W}^{\ast}_{\infty}} \le (2r)^{-1}\sum_{z \in \Z^3 }G^{\ast}(B_{zr,3r})h(B_{zr,2r}).
  \end{aligned}
\end{align}
where $G^{\ast}$ is defined like $G$, except that in this instance the initial velocity is chosen from $-\Omega_{v_0}$ rather than fixed to be $v_0$.


\begin{lemma} \label{lem:compbounds}
  The following bounds hold for some $C<\infty$ and any $v \in \Omega$

  \begin{align*}
    & \sum_{z \in \Z^3} K(B_{zr,3r}) M(B_{zr,2r}) \le Cr^3 ,
    & \sum_{z \in \Z^3} L_{v}(B_{zr,3r}) M(B_{zr,2r}) \le Cr^3 \\
    & \sum_{z \in \Z^3} K(B_{zr,3r}) L_v(B_{zr,2r}) \le Cr^3 ,
    & \sum_{z \in \Z^3} L_{v}(B_{zr,3r}) L_{w}(B_{zr,2r}) \le Cr^2.
  \end{align*}
  for $v \neq w \in \Omega$
\end{lemma}

\begin{proof}
  The following bounds follow immediately from the definitions of $K,M,$ and $L_v$

  \begin{align}
    \begin{aligned} \label{straightforward}
    & K(B_{zr,3r}) \le Cr^3 ,\\
    & M(B_{zr,3r}) \le Cr^3 e^{-cr\abs{z}},\\
    & L_v(B_{zr,3r}) \le Cr^3\delta_{0,z}  + C r \one \{ \exists t > 0  : vt \cap B_{zr,3r}\}(1-\delta_{0,z}) e^{-cr\abs{z}}.\\
    \end{aligned}
  \end{align}
  
  From here
  \begin{align*}
    \sum_{z \in \Z^3} K(B_{zr,3r}) M(B_{zr,2r}) &\le Cr^6\sum_{z \in \Z^3}e^{-cr\abs{z}}\\
    &\le Cr^3\int_{\R^3}e^{-c\abs{z}}dz \le Cr^3 
  \end{align*}
  where we use a Riemann integral to go from the first line to the second. Likewise

  \begin{align}
    \begin{aligned} \label{KL bound}
    \sum_{z \in \Z^3}K(B_{zr,3r})L_v(B_{zr,2r}) & \le Cr^{6} + Cr^4\sum_{z \in (\Z^3)^{\ast}} \one\{\exists t >0 : vt \cap B_{zr,3r}\}e^{-cr\abs{z}}\\
                                                & \le Cr^6 + C' r^4 \sum_{z=1}^\infty e^{-cr \abs{vz}}\\
                                                & \le Cr^6 + Cr^3 \int_0^{\infty} e^{-c\abs{vt}}dt \le Cr^3
    \end{aligned}
  \end{align}
  where from the first line to the second we approximate the points $zr \in r\Z^3$ close to the line $vt$ by the points $rvz$ for $z\in \Z$.

  Similarly

  \begin{equation*}
    \sum_{z \in \Z^3} L_{v}(B_{zr,3r}) M(B_{zr,2r}) \le  Cr^6 + Cr^4 \sum_{z \in (\Z^3)^*}   \one \{ \exists t > 0  : vt \cap B_{zr,3r}\} e^{-2cr\abs{z}}
  \end{equation*}
  the bound then follows as it did in \eqref{KL bound}.

  Finally,

  \begin{align*}
    \sum_{z \in \Z^3}L_v(B_{zr,3r})L_w(B_{zr,2r}) & \le Cr^2\sum_{z \in (\Z^3)^{\ast}} \one\{\exists t >0 : vt \cap B_{zr,3r}\}\one\{\exists t >0 : wt \cap B_{zr,3r}\}e^{-2cr\abs{z}}\\
      &\le Cr^2e^{-cr} \le Cr^2,\\
  \end{align*}
  since $v \neq w$ only finitely many $z \in \Z$ contribute to the sum, from which the second line follows.

\end{proof}

Note that Lemma \ref{lem:greenbounds} is stated for $G$ and $H$ and not $G^{\ast}$ and $H^{\ast}$. However similar bounds hold for the backwards excursions. Thus (omitting these details), we use Lemma \ref{lem:greenbounds} to insert Lemma \ref{lem:compbounds} into \eqref{W by Hg} to get:


\begin{proposition} \label{prop:inter-excursion}
  There exists a constant $C>0$ such that for all $j\ge 1$

  \begin{equation}
    \probab{\wh W_j} \le Cr \qquad , \qquad \probab{\wt W_j} \le Cr.
  \end{equation}

\end{proposition}

\subsection{Mismatches  within one Excursion}
\label{ss:Mismatches}

Define the following indicator functions

\begin{align}
  \begin{aligned}\label{eta}
    &\wh\eta_j = \wh{\eta}(y_{j-2},y_{j-1},y_j) := \one \left\{ \min_{ 0 \le t \le \xi_{j-2}} \left(tu_{j-2}+y_{j-1}+\beta_{j-1}\right) \in \cQ_r \right\}\\
    &\wt{\eta}_j = \wt{\eta}(y_{j-2},y_{j-1},y_j) := \one \left\{ \min_{ 0 \le t \le \xi_{j}} \left(y_{j-1}+t u_{j}-\beta_{j-2}\right)    \in  \cQ_r \right\}\\
    & \eta_j := \max\{ \wt\eta_j , \wh\eta_j\}
  \end{aligned}
\end{align}
In words, $\wh\eta_j$ is the event that the $(j-1)$-labelled collision is shadowed by the immediately preceding path (i.e a \emph{direct} shadowing event). And $\wt\eta_j$ is the event that during the $j^{th}$ path segment there is a recollision with the immediately preceding obstacle (i.e a \emph{direct} recollision) - see the left hand side of Figure \ref{fig:direct}.


\begin{lemma} \label{lem:direct}
  For any $i,j < \gamma$ with $i \neq j$ there exists a $C< \infty$ such that
  \begin{align}
    \expect{\eta_j} & \le C r   \label{directbnd}\\
    \expect{\eta_j \eta_i } & \le Cr^2 \label{directbnd 2}
  \end{align}
\end{lemma}

\noindent \eqref{directbnd 2} is not needed to prove the result for $T=o(r^{-1})$ however will be used to prove Theorem \ref{thm:main-technical}.

\begin{proof}[Proof of Lemma \ref{lem:direct}]
  Suppose $u_{j-2} = U$. Then throughout the two subsequent collisions we know for some $i=1,2,3$ - $(u_{j-1})_i = (u_j)_i = U_i$ (i.e one coordinate of the velocity remains unchanged). Thus to (directly) recollide  with $Y_{j-2}^{\prime}+\cQ_r$ we require $\xi_{j-1}<Cr$ which implies \eqref{directbnd}. The same is true for shadowing events, that is $\wh{\eta}_j =1$ implies $\xi_{j-1} > Cr$ for some constant.

  \eqref{directbnd 2} follows for the same reason. Suppose $i \neq j$, then for $\eta_j \eta_i = 1$, requires $\max \{\xi_{j-1}, \xi_{i-1}\} < Cr$ for some constant. As these are independent exponentials \eqref{directbnd 2} is immediate.
  
\end{proof}

Lemma \ref{lem:direct} controls the probability of a \emph{direct} mismatch. However we also need to control indirect mismatches. To that end define
\begin{align}
  \begin{aligned} \label{indirect def}
    &\wh\eta_j^{o}   := \one \left\{ \min_{ 0 \le t \le \tau_{j-3}}\left(  Y(t) - Y_{j-1}^{\prime}\right) \in  \cQ_r \right\}\\
    &\wt{\eta}^{o}_j := \one \left\{ \min_{ \tau_{j-1} \le t \le \tau_{j}}\left(\min_{ 0 \le k \le j-3} \left( Y(t) - Y_{k}^{\prime} \right)\right) \in \cQ_r \right\}\\
    & \eta_j^{o} := \max\{ \wt\eta_j^o , \wh\eta_j^o\}
  \end{aligned}
\end{align}
In words $\wh \eta_j^{o}$ is the indicator that an \emph{indirect} (virtual) shadowing event occurs and $\wt \eta_j^{o}$ is the event an \emph{indirect} (virtual) recollision occurs. That is a mismatch which involves more than the immediately preceding obstacle or path. 


\begin{lemma}\label{lem:indirect}
  For any $3<j\le\gamma$ there exists a constant $C>0$ such that

  \begin{align}
    \expect{\eta_j^{o}} & \le C \gamma^2 r^2 \label{indirectbnd}
  \end{align}
\end{lemma}

\begin{proof}[Proof of Lemma \ref{lem:indirect}]
  Under time reversal Markovian flight processes remain Markovian flight process while recollisions become shadowed events. Hence recollisions and shadowing events happen with the same probability and thus we may restrict to proving the statement for recollisions.

  By the union bound
  \begin{equation} \label{union}
    \expect{\wt{\eta}_j^o} \le \sum_{k\le j-3} \probab{\left\{\min_{ \tau_{j-1} \le t \le \tau_j}  \left(Y(t) - Y_k^{\prime}\right) \in\cQ_r\right\} }.
  \end{equation}
  Write $\mathscr{A}_k = \{\min_{ \tau_{j-1} \le t \le \tau_j} \left( Y(t) - Y_k^{\prime} \right) \in\cQ_r\}$ - the event there is a indirect recollision after $k-1$ fresh collisions. To have an indirect recollision, requires at least three distinct velocities along the path, thus

  \begin{equation*}
    \probab{ \mathscr{A}_k } = \probab{\mathscr{A}_k \cap \{\exists i \in [k+1 , j-2] : u_i \neq u_j,u_{j-1}\}} \label{ui new}.
  \end{equation*}

  Moreover at each collision exactly one of the velocity coordinates changes sign. Hence we know $u_j$ and $u_{j-1}$ differ by a sign change in one coordinate therefore the event in the right hand side of \eqref{ui new} implies there is a third velocity which is linearly independent of $u_j$ and $u_{j-1}$. Therefore

  \begin{equation*}
    \mbox{\eqref{ui new}} = \probab{\mathscr{A}_k \cap \{\exists i \in [k+1, j-2]\; : \; u_i, u_j,u_{j-1} \mbox{ lin. ind.}\}}
  \end{equation*}
  Moreover note that if we fix $i$

  \begin{equation*}
    \cA_k = \{ \min_{ 0 \le t \le \xi_j} \left( \xi_i u_i + \xi_{j-1} u_{j-1} + t u_j - s_i \right) \in \cQ_r \}
  \end{equation*}
  where

  \begin{equation*}
    s_i = \sum_{\substack{l = k+1 \\ l \neq i}}^{j-2} u_l \xi_l.
  \end{equation*}
  Let $B_i$ denote the event $u_i, u_{j-1}, u_j$ are linearly independent. In this case

  \begin{align*}
    \probab{\cA_k} &\le \sum_{i=k+1}^{j-2} \probab{B_i \cap \cA_k  }\\
                   &\le \sum_{i=k+1}^{j-2} \expect{\condprobab{B_i \cap \{\min_{0 \le t \le \xi_j} \left(\xi_i u_i + \xi_{j-1}u_{j-1} + tu_j - s_i\right) \in \cQ_r\} }{s_i}}.
  \end{align*}
  Lemma \ref{lem:lin ind} (below) implies that the probability inside the expectation is bounded by $Cr^2$. As $j-2- k \le \gamma$ this implies

  \begin{align*}
    \probab{\cA_k} \le C\gamma r^2.
  \end{align*}
  Inserting this into \eqref{union} then implies \eqref{indirectbnd}

\end{proof}

\begin{lemma} \label{lem:lin ind}
  Suppose $U_1,U_2,U_3\in \Omega$ are linearly independent and $\xi_1,\xi_2,\xi_3 \sim EXP(1)$ are i.i.d exponentials. Then there exists a constant $C<\infty$ such that for any $s \in \R^3$

  \begin{equation} \label{lin ind}
    \probab{\min_{ 0 \le  t \le \xi_3} \left(U_1\xi_1 + U_2\xi_2 +U_3 t-s\right) \in  \cQ_r} \le Cr^2.
  \end{equation}
\end{lemma}

\begin{proof}
  We can assume

  \begin{align*}
    U_1 = (\nu_1,\nu_2,\nu_3) \quad, \quad U_2 = (-\nu_1,\nu_2,\nu_3) \quad, \quad U_3 = (-\nu_1,-\nu_2,\nu_3)
  \end{align*}
  in which case for any $t \le \xi_3$

  \begin{equation}
    U_1\xi_1 + U_2\xi_2 +U_3 t = ((\xi_1  -\xi_2 - t)\nu_1, (\xi_1 + \xi_2 - t)\nu_2, (\xi_1 + \xi_2 + t)\nu_3).
  \end{equation}
  Therefore the event on the left hand side of \eqref{lin ind} is the event that there exists a $t \le \xi_3$ satisfying the system of inequalities

  \begin{align*}
    s_1 - \frac{r}{2}\quad &\le \quad (\xi_1  -\xi_2 - t)\nu_1 \quad  \le \quad  s_1 - \frac{r}{2}\\
    s_2 - \frac{r}{2}\quad &\le \quad (\xi_1 + \xi_2 - t)\nu_2 \quad \le \quad s_2 - \frac{r}{2}\\
    s_3 - \frac{r}{2}\quad &\le \quad (\xi_1 + \xi_2 + t)\nu_3  \quad \le \quad s_3 - \frac{r}{2}\\
  \end{align*}
  solving these equations, we find that regardless of $t$ there exist $c_1,c_2,C_1,C_2$ such that

  \begin{equation*}
    \xi_1 \in [c_1 -C_1 r, c_1 +C_1 r] \qquad, \qquad 
    \xi_2 \in [c_2 -C_2 r, c_2 +C_2 r]
  \end{equation*}
  since $\xi_1$ and $\xi_2$ are i.i.d exponentials \eqref{lin ind} follows immediately.
\end{proof}


\section{Beyond the Na\"{i}ve Coupling}
\label{s:Beyond the Naive Coupling}

In the following sections we extend the results of Section \ref{s:No mismatches} to times on the order $o(r^{-2})$. In order to reduce the amount of notation we will use the same notation for the \emph{analogous} objects and will give the redefinitions explicitly. Recall the definition of the process $\{t \mapsto Z(t)\}$ given in Subsection \ref{ss:Joint Construction}. We will split the process $\{t \mapsto Z(t)\}$ into legs (similar to the excursions of the previous section).

\subsection{Legs}
\label{ss:Legs - beyond}

Similar to Subsection \ref{ss:Excursions} we split $t \mapsto Z(t)$ into legs. However to ensure that the different legs are independent we impose the restriction that each leg begins and ends with two path segments of length greater than $1$. Let $\wt\xi_n = \wt\tau_n - \wt\tau_{n-1}$ for all $n \ge 1$. Let

\begin{equation}
  \gamma := \min \{ i > 1 : \wt\xi_{i-1}, \wt\xi_{i}, \wt\xi_{i+1}, \wt\xi_{i+1} > 1 \;, \;\;\wt w_{i+1}=\wt w_1 = v_0 \}.
\end{equation}
Note that the condition on $\wt\xi_i$ implies that $\gamma \in \{2\} \cup \{5,\dots\}$. If we define $\theta := \sum_{i=1}^\gamma\wt\xi_i$ then

\begin{align}\label{gammabnd}
  \probab{\gamma> s} \le C e^{-cs} \qquad , \qquad \probab{\theta >s} \le Ce^{-cs}.
\end{align}
The definition of a pack is then similar to Subsection \ref{ss:Excursions}: a \emph{pack} is a collection

\begin{equation*}
  \varpi := \left( \gamma ; \{ \wt\xi_i \}_{i=1}^{\gamma}, \{\wt{\beta}_i\}_{i=1}^{\gamma}, \{\wt w_i \}_{i=1}^{\gamma}     \right),
\end{equation*}
 Given a pack we consider the process $t \mapsto Z(t)$ associated to it via the rules set forth in Subsection \ref{ss:Joint Construction} and call such a segment a \emph{leg}.  Note that, in order to have a direct mismatch at step $n$ requires that $\wt{\xi}_{n-1}< Cr$ for some constant $C<\infty$. Hence the beginning and end of a leg are Markovian steps.

Furthermore given a pack $\varpi$ a \emph{backwards leg} is defined to be 

\begin{equation*}
  (\theta;Z^{\ast}(t); 0 \le t \le \theta )
\end{equation*}
where

\begin{equation*}
  Z^{\ast}(t) = Z(\theta-t , \varpi^{\ast}) - \overline{Z}(\varpi^{\ast})
\end{equation*}
(we use the notation $Z(t,\varpi)$ to denote the forward forgetful process built from the pack $\varpi$) where 

\begin{equation*}
  \varpi^{\ast} := (\gamma ; \{\wt{\xi}_{\gamma-j}\}_{j=0}^{\gamma-1}, \{\wt{\beta}_{\gamma -j}\}_{j=0}^{\gamma-1}, \{\wt w_{\gamma-j}\}_{j=0}^{\gamma-1})
\end{equation*}
As before denote

\begin{equation*}
  Z^{\ast}_j := Z^{\ast}(\wt\tau_j), \quad 0\le j \le \gamma \qquad, \qquad \overline{Z^{\ast}} = Z^{\ast}_{\gamma}.
\end{equation*}
Note the processes $t \mapsto Z(t)$ and $t \mapsto Z^{\ast}(t)$ do not have the same distribution.

\subsection{Concatenation}
\label{ss:Concatenation - beyond}

Let $\varpi_n = \left(\gamma_n; \;\{ \wt\xi_{n,j} \}_{j=1}^{\gamma_n} \{\wt{\beta}_{n,j}\}_{j=1}^{\gamma_n}, \{\wt w_{n,j}\}_{j=1}^{\gamma_n}      \right)$, $n\ge 1$,  be a sequence of i.i.d \emph{packs} and consider the associated forwards legs  $(Z_n(t): 0\le t\le \theta_n)$, $(Z_{n,j}: 1\le j\le\gamma_n)$ and backwards legs $(Z^*_n(t): 0\le t\le \theta_n)$, $(Z^*_{n,j}: 1\le j\le\gamma_n)$.   

To construct the concatenated forward and backward  processes $t\mapsto Z(t)$, $t\mapsto Z^*(t)$, $0\le t<\infty$, define for $n\in\Z_+$ and $t\in\R_+$
\begin{align}
  \label{Concat labels}
  \begin{aligned}
  &\Gamma_n :=   \sum_{k=1}^n \gamma_k,      
  &&   \nu_n :=   \max\{m:\Gamma_m\le n\},       
  && \{n\} := n-\Gamma_{\nu_n},  \\
  &   \Theta_n := \sum_{k=1}^n \theta_k,        
  && \nu_t := \max\{m:\Theta_m<t\},        
  &&   \{t\} := t-\Theta_{\nu_t}.
  \end{aligned}
\end{align}
The concatenated (multi-leg) forward and backward $Z$-processes are 
\begin{align}
  \label{Xi-walk}
  \begin{aligned}
    &  \Xi_n  :=  \sum_{k=1}^n \overline{Z}_k, \qquad
    && \qquad Z_n := \Xi_{\nu_n} + Z_{\nu_n+1, \{n\}}, \qquad
    && \qquad Z(t) := \Xi_{\nu_t} + Z_{\nu_t+1}(\{t\}), \\
    &  \Xi^*_n:=  \sum_{k=1}^n \overline{Z}^*_k, \qquad
    && \qquad Z^*_n := \Xi^*_{\nu_n} + Z^*_{\nu_n+1, \{n\}}, \qquad
    && \qquad Z^*(t) := \Xi^*_{\nu_t} + Z^*_{\nu_t+1}(\{t\}).
  \end{aligned}
\end{align}

\subsection{Mismatches in a Leg}
\label{ss:Mismatches in an Leg}

Let $\varpi = (\gamma; \{\wt\xi_j\}_{j=1}^\gamma,\{\wt{\beta}_j\}_{j=1}^\gamma,\{\wt w_j\}_{j=1}^\gamma)$ be a pack. Let $u \in \Omega_{v_0}$ a velocity and $\beta_0 \in B(u,v_0)$ an impact parameter.

Let $t \mapsto \mathscr{X}(t)$ be the wind-tree process coupled to the pack $\varpi$. That is, given the processes $t \mapsto Y(t)$ and $t \mapsto Z(t)$ follow the rules in Subsection \ref{ss:Joint Construction} until time $\tau_{\gamma}$.

Consider the jointly realized triple $((Y(t), \mathscr{X}(t),Z(t)): 0^- < t < \theta^+)$ - a Markovian flight process, a wind-tree exploration process and a forgetful process all coupled to $\varpi$. The time interval $0^- < t < \theta^+$ indicates that the velocity immediately prior to the position at $0$ is $u$, there is a collision with a scatterer at $\beta_0$, and at $\theta^+$ the velocity of $Y$ and $Z$ is $w$.


\begin{proposition}\label{prop:intra-leg}
  There exists a $C<\infty$ such that for all $w \in \Omega$ and $u \in \Omega_{w} $ and $\beta_0 \in B(u,w)$

  \begin{equation}\label{intra-exc}
    \probab{\mathscr{X}(t) \not\equiv Z(t) : 0^- < t < \theta^+} \le r^2  .
  \end{equation}

\end{proposition}
\noindent This proposition will be proved in Section \ref{s:Proof of Intra}.

\subsection{Inter-Leg Mismatches}
\label{ss:Inter-Leg Mismatches}

Consider a forgetful process $t \mapsto Z(t)$ built from legs. Define the following events

\begin{align}
  \label{What-Wtilde}
  \begin{aligned}
    \wh W_j  &:= \big\{ \{ Z(t) - Z^{\prime}_k : && 0 < t < \Theta_{j-1}, && \Gamma_{j-1}< k \le \Gamma_j\} \cap \cQ_r \neq \emptyset \big\}, \\
    \wt W_j  &:=  \big\{  \{Z^{\prime}_k-Z(t):  && 0\le k< \Gamma_{j-1}, && \Theta_{j-1}< t < \Theta_j \} \cap \cQ_r \neq \emptyset \big\},
  \end{aligned}
\end{align}
i.e $\wh{W}_j$ is the event a collision during the $j^{th}$ leg is (virtually) shadowed by a path segment in a previous leg. $\wt{W}_j$ is the event that during the $j^{th}$ leg the process (virtually) collides with an obstacle placed during a previous leg.


\begin{proposition}\label{prop:inter-leg - beyond}
  There exists a $C<\infty$ such that for all $j \ge 1$,

  \begin{equation}
    \probab{\wt{W}_j } \le Cr^{2} \qquad , \qquad \probab{\wh{W}_j } \le Cr^{2}.
  \end{equation}

\end{proposition}

\noindent The proof of this proposition is the content of Section \ref{s:Proof of Inter}.


\section{Proof of Proposition \ref{prop:inter-leg - beyond}}
\label{s:Proof of Inter}

The proof of Proposition \ref{prop:inter-leg - beyond} follows the similar lines to that of Proposition \ref{prop:inter-excursion}. However as we have redefined legs we shall go through the full proof. In this section we redefine the Green's functions $g,h,G,$ and $H$.

\subsection{Occupation Measures} 
\label{ss:Occupation Measures - beyond}

Let $t \mapsto Z(t)$ be a forward forgetful process with initial velocity $v_0$ and $t \mapsto Z^{\ast}(t)$ a backward process with initial velocity in $\Omega_{-\wt w_1}$ (distributed according to $\mathrm{m}_{-v_0}$). Define the events

\begin{align*}
  \wh W_j^* &:=  \big\{ \{ Z^*(t) - Z^{\prime}_k : && 0 < t < \Theta_{j-1}, && 0 < k \le \gamma\} \cap \cQ_r \neq \emptyset \big\}, \\
  \wt W_j^* &:=  \big\{  \{Z^{*\prime}_k-Z(t): && 0 < k \le \Gamma_{j-1},&& 0 < t < \theta \} \cap \cQ_r \neq \emptyset \big\}, \\
  \wh W_\infty^* &:=  \big\{ \{ Z^*(t) - Z^{\prime}_k : && 0 < t < \infty, && 0 < k \le \gamma\} \cap \cQ_r \neq \emptyset \big\}, \\
  \wt W_\infty^* &:=  \big\{  \{ Z^{*\prime}_k-Z(t):  && 0 < k < \infty, && 0 < t < \theta \} \cap \cQ_r \neq \emptyset \big\}.
\end{align*}
The same calculation as \eqref{hat-tilde-Ws}, \eqref{pwhtbound-for-Y}, and \eqref{W by Hg} implies

\begin{align} 
  \begin{aligned}
  \label{Wj Winfty}
    &\probab{\wt{W}_j} \le  \probab{\wt{W}^{\ast}_{\infty}} \le (2r)^{-1} \sum_{z \in \Z^3 }H^{\ast}(B_{zr,3r})g(B_{zr,2r}), \\
    & \probab{\wh{W}_j} \le \probab{\wh{W}^{\ast}_{\infty}} \le (2r)^{-1}\sum_{z \in \Z^3 }G^{\ast}(B_{zr,3r})h(B_{zr,2r}),
  \end{aligned}
\end{align}
where the right hand side is in terms of the following Green's functions: for $A \subset \R^3$ 

\begin{align*}
  &g(A)    :=  \expect{ \abs{ \{1\le k\le \gamma: Z_k\in A\} } },   
  &&g^*(A)  :=  \expect{ \abs{ \{1\le k\le \gamma: Z^*_k\in A\} } }, \\ 
  &h(A)    :=  \expect{ \abs{ \{0< t \le \theta: Z(t)\in A\} } }, 
  &&h^*(A)  :=  \expect{ \abs{ \{0< t \le \theta: Z^*(t)\in A\} } },   \\
  &R^*(A)  :=  \expect{ \abs{ \{1\le n< \infty: \Xi^*_n\in A\} } },  \\
  &G^*(A)  :=  \expect{ \abs{ \{1\le k< \infty: Z^*_k\in A\} } },    
  &&H^*(A)  :=  \expect{ \abs{ \{0< t< \infty: Z^*(t)\in A\} } }.
\end{align*}
Note that 
\begin{align}
  \label{convolution-1}
  \begin{aligned}
    G^*(A) &= g^*(A) + \int_{\R^{3}} g^*(A-x)R^*(dx), \\
    H^*(A) &= h^*(A) + \int_{\R^{3}} h^*(A-x)R^*(dx).
  \end{aligned}
\end{align}

\subsection{Bounds}

\label{ss:Bounds - beyond}

\begin{lemma}
  \label{lem:greenbounds - beyond}
  The following bounds hold for any Borel set $A \subset \R^3$

  \begin{align}
    \label{g bound 2}
    &  g(A)  \le M(A) + \wt{L}_{v_0}(A),  
    && g^{\ast}(A)\le M(A) + \wt{L}^{\perp}_{v_0}(A),\\
    \label{h bound 2}
    &  h(A)  \le M(A) + L_{v_0}(A),  
    && h^{\ast}(A)\le M(A) + L^{\perp}_{v_0}(A),\\
    \label{R bound 2}
    &  R^*(A) \le K(A) + \wt{L}_{v_0}^{\perp}(A),\\
    \label{G bound 2}
    &  G^*(A) \le K(A) + \wt{L}_{v_0}^{\perp}(A),    
    && H^{\ast}(A) \le K(A) + L^{\perp}_{v_0}(A),
  \end{align}
  where $K$, $L_{v_0}$, and $M$ are as defined in Lemma \ref{lem:greenbounds} and

  \begin{align*}
    &L^{\perp}_{v_0}(A) := C \sum_{w \in \Omega_{-v_0}} \int_0^{\infty} \one\{tw \cap A\}e^{-ct}dt,\\
    &\wt{L}_{v_0}(A) := C \int_{1}^{\infty} \one\{tv_0 \cap A\}e^{-ct}dt, && \wt{L}^{\perp}_{v_0}(A) := C \sum_{w \in \Omega_{-v_0}}\int_1^{\infty}\one\{tw \cap  A\}e^{-ct}dt.
  \end{align*}

\end{lemma}

\begin{proof}
  The proof of this Lemma follows the same lines as the proof of Lemma \ref{lem:greenbounds} however the legs in this section are conditioned to have the first step longer than $1$. \eqref{R bound 2} follows from the fact that the steps of $\Xi_n^{\ast}$ are i.i.d with exponentially decaying tails and the density of each step is bounded by $g^{\ast}(dx)$.

  To bound $g(A)$ write:

  \begin{align*}
    g(A)   &= \int_{\R^3} g_2(A-x) g_1(dx),\\
    g_1(A) &:= \probab{Z_1 \in A} = C \int_1^{\infty} \one\{tv_0 \in A\}e^{-t}dt,\\
    g_2(A) &:= \expect{\abs{\{1 \le k \le \gamma_1 : Z_k - Z_1 \in A \}}}.
  \end{align*}
  This follows since $Z_k-Z_1$ is independent of $Z_1$ for every $k\ge 2$. \eqref{g bound 2} then follows in the same way as did \eqref{g bound} in Lemma \ref{lem:greenbounds} from the bounds

  \begin{align*}
    g_2(\{ x : \abs{x} >s \}) \le C e^{-cs} \qquad , \qquad g_2(\R^3) = \expect{\gamma} < \infty.
  \end{align*}
  For $g^*(A)$ write

  \begin{align*}
    g^{\ast}(A) &= \expect{ \abs{\{ 1 \le k \le \gamma_1 : Z_k^{\ast} \in A\}}}\\
               &\le \sum_{w \in \Omega_{-v_0}} \condexpect{\abs{\{ 1 \le k \le \gamma_1 : Z_k^{\ast} \in A \}}}{\wt w_1^{\ast} = w} = : \sum_{w \in \Omega_{-v_0}}g_w^{\ast}(A),
  \end{align*}
  where $\wt w_1^{\ast} := \dot{Z}^{\ast}(0^+)$. As for $g(A)$ we now split

  \begin{align*}
    g_w^{\ast}(A)     &= \int_{\R^3} g^{\ast}_{2,w}(A-x) g_{1,w}^\ast(dx) \\
    g_{1,w}^{\ast}(A) &:= \condprobab{Z_1^{\ast} \in A}{\wt w_1^{\ast} = w}\\
    g_{2,w}^{\ast}(A) &:= \condexpect{ \abs{ \{ 1 \le k \le \gamma_1 : Z_k - Z_1 \in A \}}}{\wt w_1^{\ast} = w}
    \end{align*}
    Our bound for $g^{\ast}(A)$ now follows the same lines as for $g(A)$. $h^{\ast}(A)$ is very similar.

    The bounds on $G^{\ast}$ and $H^{\ast}$ follow by inserting the bounds for $g^{\ast},h^{\ast}, R^{\ast}$ into \eqref{convolution-1}.

\end{proof}

\subsection{Computations}


\begin{lemma} \label{lem:computation - beyond}
  The following bounds hold for some $C<\infty$ and $r$ small enough 

  \begin{align*}
    & \sum_{z \in \Z^3} \wt{L}^{\perp}_{v_0}(B_{zr,3r}) L_{v_0}(B_{zr,2r}) =  0,  
    && \sum_{z \in \Z^3} L^{\perp}_{v_0}(B_{zr,3r}) \wt{L}_{v_0}(B_{zr,2r}) =  0, \\
    & \sum_{z \in \Z^3} K(B_{zr,3r}) \wt{L}_{v_0}(B_{zr,2r}) \le  Cr^3,     
    && \sum_{z \in \Z^3} \wt{L}^{\perp}_{v_0}(B_{zr,3r})M(B_{zr,2r})  \le  Cr^3,\\  
    & \sum_{z \in \Z^3} L^{\perp}_{v_0}(B_{zr,3r})M(B_{zr,2r})  \le  Cr^3.  \\
  \end{align*}

\end{lemma}

\begin{proof}
  These bounds follow by observing

  \begin{align}
    \begin{aligned} \label{st8f 2}
    &\wt{L}_{v_0}(B_{zr,3r}) \le C\one \{\exists t\ge 1 : B_{zr,3r} \cap v_0 t\}re^{-cr\abs{z}},\\
    &\wt{L}_{v_0}^{\perp}(B_{zr,3r}) \le C\sum_{w \in \Omega_{-v_0}}\one \{ \exists t \ge 1 :B_{zr,3r} \cap w t\}re^{-cr\abs{z}},\\
    & L_{v_0}^{\perp}(B_{zr,3r}) \le C\delta_{0,z}r^3  + C\sum_{w \in \Omega_{-v_0}}\one\{\exists t \ge 3r : B_{zr,3r} \cap wt \} r e^{-cr\abs{z}},
      \end{aligned}
  \end{align}
  and \eqref{straightforward}. With that the first two bounds are trivial. The third bound follows from:

  \begin{align*}
    \sum_{z \in \Z^3} K(B_{zr,3r})\wt{L}_{v_0}(B_{zr,2r}) &\le Cr^6 + Cr^4 \sum_{w\in \Omega_{-v_0}} \sum_{z \in \ (\Z^3)^{\ast}} \one \{ \exists t \ge 3r : B_{zr,3r} \cap wt\} e^{-cr\abs{z}}\\
                                                      &\le Cr^6 + Cr^4 \sum_{w \in \Omega_{-v_0}} \sum_{z \in \Z^{\ast}} e^{-cr\abs{vz}} \le Cr^3,
  \end{align*}
  where in the last line we approximate the sum by an integral in the same way as we did in \eqref{KL bound}.
  
  Note that by \eqref{st8f 2}

  \begin{align*}
     \sum_{z \in \Z^3} \wt{L}^{\perp}_{v_0}(B_{zr,3r})M(B_{zr,2r})  \le    
     \sum_{z \in \Z^3} L^{\perp}_{v_0}(B_{zr,3r})M(B_{zr,2r}). 
  \end{align*}
  Moreover by \eqref{straightforward} and \eqref{st8f 2}

  \begin{align*}
    \sum_{z \in \Z^3} L^{\perp}_{v_0}(B_{zr,3r})M(B_{zr,2r}) \le Cr^6 + Cr^4 \sum_{w \in \Omega_{-v_0}} \one \{\exists t \ge 3r : B_{zr,3r} \cap wt\} e^{-2cr\abs{z}} \le Cr^3.
  \end{align*}

\end{proof}

\begin{proof}[Proposition \ref{prop:inter-leg - beyond}]
  The proof of Proposition \ref{prop:inter-leg - beyond} follows by inserting the bounds in Lemma \ref{lem:greenbounds - beyond} into \eqref{Wj Winfty} and then applying Lemma \ref{lem:computation - beyond}.

\end{proof}

\section{Proof of Proposition \ref{prop:intra-leg}}
\label{s:Proof of Intra}

In the setting of Section \ref{ss:Mismatches in an Leg} the proof of Proposition \ref{prop:intra-leg} will follow from considering the following indicator functions

\begin{align}
  \begin{aligned}
  &\wt\eta_j := \one \left\{ \min_{ \wt\tau_{j-1} < t < \wt\tau_j} \left( Z(t) -  Z_{j-2}^\prime \right) \in \cQ_r \right\}\\
  &\wh\eta_j:= \one \left\{ \min_{ \wt\tau_{j-3} < t < \wt\tau_{j-2}} \left( Z(t) - Z(\wt\tau_{j-1}) - \wt\beta_{j-1}\right) \in  \cQ_r \right\}\\
  &\eta_j := \max\{\wt\eta_j, \wh\eta_j\}
  \end{aligned}
\end{align}
In particular, $\eta_j$ is the probability of a mismatch for the $Z$-process in immediately before the $j^{th}$ leg. It is important to note, the simple geometric fact (which follows simply from the fact that the collision angles are bounded) that $\eta_j^* =1$ implies $\wt{\xi}_{j-1} < Cr$ for some constant $C< \infty$. This fact will make the geometric estimates vastly easier than for the Lorentz gas, where the equivalent statement is false.

The following statements will provide the proof of Proposition \ref{prop:intra-leg}

\begin{align}
  \label{etamore}
  & \probab{\{\cX(t)\not\equiv Z(t): 0^-\le t\le \theta^+\} \cap \{\sum_{j=1}^{\gamma} \eta_j > 1\}}\le Cr^2, \\
  \label{etazero}
  & \probab{\{\cX(t)\not\equiv Z(t): 0^-\le t\le \theta^+\} \cap \{\sum_{j=1}^{\gamma}  = 0\}}\le Cr^2, \\
  \label{etaone}
  & \probab{\{\cX(t)\not\equiv Z(t): 0^-\le t\le \theta^+\} \cap \{\sum_{j=1}^{\gamma}  = 1\}}\le Cr^2.
\end{align}

\subsection{Proof of \eqref{etamore}}
\label{ss:Proof of etamore}

The simple geometric fact stated in the previous section implies

\begin{align*}
  \probab{\sum_{j=1}^{\gamma} \eta_j > 1 } \le \frac{\gamma^2}{2} \max_{1\le j<k\le \gamma}\probab{\eta_j\eta_k = 1 }
     \le C \gamma^2 r^2.            
\end{align*}
\eqref{etamore} now follows from the exponential tail bounds \eqref{gammabnd}.

\qed

\subsection{Proof of  \eqref{etazero}}

On $\left\{\sum_{j=1}^{\gamma} \eta_j = 0\right\}$, the process $\{t \mapsto Z(t)\}$ is distributed like a Markovian flight process. Hence the event in \eqref{etazero} can be written

\begin{align*}
  \{\cX(t)\not\equiv Z(t): 0^-\le t\le \theta^+\} \cap \left\{\sum_{j=1}^{\gamma} \eta_j = 0\right\} = \{\exists \; 3\le j \le \gamma : \eta_j^o=1\} \cap \{\sum_{j=1}^{\gamma} \eta_j = 0 \}
\end{align*}
where $\eta_j^o$ is the indicator of an indirect mismatch, as defined in \eqref{indirect def}.  Therefore using Lemma \ref{lem:indirect}

\begin{align*}
  \probab{\{\cX(t)\not\equiv Z(t): 0^-\le t\le \theta^+\} \cap \{\sum_{j=1}^{\gamma} \eta_j = 1\}} &\le \probab{\{\exists \; 3 \le j \le \gamma : \eta^o_j =1 \}} \\
                  &\le \gamma \max_{3 \le j \le \gamma }\probab{\eta_j^o =1 }\\
                  &\le C\gamma^3 r^2.
\end{align*}
Thus \eqref{etazero} again follows from the exponential tail bounds \eqref{gammabnd}.

\qed






\subsection{Proof of \eqref{etaone}}

Given a $\gamma \in \{2\}\cup \{5,\dots\}$, a signature $\underline{\epsilon}$ (recall the definition of a signature given at the end of Subsection \ref{ss:Markovian Flight Process}) compatible with the definition of a pack, and a fixed label $3 < k < \gamma$. Let $V_1, V_2 \in \Omega$ and let $\varpi$ be a pack with signature $\ueps$ and $\wt w_{k-2} = V_1$ and $\wt w_{k+1} = V_2$ (we assume $V_1$ and $V_2$ are compatible with this definition). 

\begin{itemize}
  \item On $0^- < t \le \wt\tau_{k-1}$ - $Z^{(k)}(t) = Y(t)$, conditioned such that $\wt w_{k-2} = V_1$.
  \item On $\wt\tau_{k-1} < t \le \wt\tau_k$ - $Z^{(k)}(t)$ is constructed like the $Z$-process, conditioned such that the final velocity is $\wt w_{k} \in  \Omega_{V_2}$
  \item On $\wt\tau_k < t < \wt\tau_{\gamma}$ - $Z^{(k)}(t)=Y(t)$ a Markovian flight process starting at $Z^{(k)}(\wt\tau_k)$, conditioned such that $\wt w_{k+1} = V_2$.
\end{itemize}
On $\{\eta_j = \delta_{j,k}: 1\le j \le \gamma\}$ - $Z^{(k)}$ is distributed like $Z$. The reason for conditioning on $V_1$ and $V_2$ is to ensure the following three parts are independent:

\begin{align}
  \begin{aligned} \label{decomp}
    & (Z^{(k)}(t) : 0^- < t \le \wt\tau_{k-3}) = (Y(t):0^- < t \le \wt\tau_{k-3}),    \\
    & (Z^{(k)}(\wt\tau_{k-3} +t) - Z^{(k)}(\wt\tau_{k-3}) : 0 \le t \le \wt\tau_{k}-\wt\tau_{k-3}),    \\
    & (Z^{(k)}(\wt\tau_k + t ) - Z(\wt\tau_k) : 0 \le t < \theta^+-\wt\tau_k ).
  \end{aligned}
\end{align}

Let $A_{a,a}^{(k)}$, $1 \le a \le 3$ be the event that the $a$-th part of the trajectory is $r$-\emph{inconsistent}. For $1 \le a< b \le 3$ we denote $A_{a,b}^{(k)}$ the event that the $a$ and $b$-th parts are $r$-\emph{incompatible}.
Therefore to prove \eqref{etaone} we will bound

\begin{align}
  \label{twelve bounds}
  \begin{aligned}
    & \max_{\underline{\epsilon}, k,V_1,V_2}\condprobab{\{\wh\eta_k  = 1\} \cap A^{(k)}_{a,b}}{\underline{\epsilon}, V_1, V_2},  \\
    & \max_{\underline{\epsilon}, k,V_1,V_2}\condprobab{\{\wt\eta_k  = 1\} \cap \{\wh\eta_k  = 0\} \cap A^{(k)}_{a,b}}{\underline{\epsilon}, V_1, V_2}, 
  \end{aligned}
  \qquad\qquad\qquad\qquad
  a,b=1,2,3.
\end{align}

\subsection{Bounds}
First notice that $A_{1,1}^{(k)}, A_{3,3}^{(k)}$ and $A_{1,3}^{(k)}$ involve only Markovian segments hence the following estimates follow readily from Lemmas \ref{lem:greenbounds}, \ref{lem:compbounds}, \ref{lem:direct}, and \ref{lem:indirect}:

\begin{align}
  \label{twelve bounds}
  \begin{aligned}
    & \max_{\underline{\epsilon}, k,V_1,V_2}\condprobab{\{\wh\eta_k  = 1\} \cap A^{(k)}_{a,b}}{\underline{\epsilon},V_1,V_2} \le C\gamma^3 r^2,  \\
    & \max_{\underline{\epsilon}, k,V_1,V_2}\condprobab{\{\wt\eta_k  = 1\} \cap \{\wh\eta_k  = 0\} \cap A^{(k)}_{a,b}}{\underline{\epsilon},V_1,V_2} \le C\gamma^3 r^2, 
  \end{aligned}
  \qquad\qquad\qquad\qquad
  a,b=1,3.
\end{align}
Therefore there remain $6$ bounds.

Note that during middle segment in \eqref{decomp} the velocity of $Z^{(k)}(t)$ is restricted to only three possible velocities. Thus one component of the velocity remains unchanged throughout this segment. Therefore the middle segment can only be $r$-inconsistent if two of the path segments are shorter than $Cr$ for some constant $C<\infty$. Thus

\begin{align}
  \label{A22}
  \begin{aligned}
    & \condprobab{ \{\wh\eta_k  = 1\} \cap A^{(k)}_{2,2}}{\underline{\epsilon},V_1,V_2 } \le  C r^2, \\
    & \condprobab{\{\wt\eta_k  = 1\} \cap \{\wh\eta_k  = 0\} \cap A^{(k)}_{2,2}}{\underline{\epsilon}, V_1,V_2 } \le  C r^2. 
  \end{aligned}
\end{align}

It remains to prove

\begin{align}
  \label{A12-A32}
  \begin{aligned}
    &\condprobab{\{\wh\eta_k  = 1\} \cap A^{(k)}_{b,2}}{\underline{\epsilon} ,V_1,V_2 } \le C\gamma r^2, \\
    & \condprobab{\{\wt\eta_k  = 1\} \cap \{\wh\eta_k  = 0\} \cap A^{(k)}_{b,2}}{\underline{\epsilon},V_1,V_2 }\le  C\gamma r^2,
  \end{aligned}
  \qquad\qquad\qquad\qquad
  b=1,3.
\end{align}
We will only prove \eqref{A12-A32} for $b=3$ as the proof for $b=1$ is the same. Given a set $A \subset \R^3$ define the following occupation measures for the third part of \eqref{decomp}

\begin{align*}
  G^{(k)}_{\ueps}(A) := & \condexpect {\#\{1\le j\le\gamma-k: Z^{(k)}(\wt\tau_{j+k}) - Z^{(k)}(\wt\tau_k)\in A\} } {\eps_{k+j}: 1\le j\le\gamma-k,V_2},\\
                        & \condexpect {\#\{1\le j\le\gamma-k: \wt{Y}(\wt\tau_j)\in A\} } {\eps_{k+j}: 1\le j\le\gamma-k,V_2},\\
  H^{(k)}_{\ueps}(A) := & \condexpect {\left|\{\tau_j \le \theta: Z^{(k)}(t) - Z^{(k)}(\wt\tau_k)\in A\}\right| } {\eps_{k+j}: 1\le j\le\gamma-k,V_2},\\
                        & \condexpect {\left|\{0\le t\le \tau_{\gamma-k}: \wt{Y}(t)\in A\}\right| } {\eps_{k+j}: 1\le j\le\gamma-k,V_2},
\end{align*}  
where $t \mapsto \wt{Y}(t)$ is a Markovian flight process with initial velocity in $\Omega_{V_2}$. Similarly
\begin{align*}
  & \wh G^{(k)}_{\ueps}(A) := \condexpect{\#\{1\le j\le3: Z^{(k)}(\wt\tau_{k-j})-Z^{(k)}(\wt\tau_{k})\in A\} \cdot \wh\eta_k}{\ueps,V_1,V_2},\\
  & \wh H^{(k)}_{\ueps}(A) := \condexpect {\abs{\{\wt\tau_{k-3}\le t\le \wt\tau_k: Z^{(k)}(t)-Z^{(k)}(\wt\tau_{k})\in A\}} \cdot \wh\eta_k} {\ueps,V_1,V_2},\\
& \wt G^{(k)}_{\ueps}(A) := \condexpect {\#\{1\le j\le3: Z^{(k)}(\wt\tau_{k-j})-Z^{(k)}(\wt\tau_{k})\in A\} \cdot \wt\eta_k \cdot (1-\wh\eta_k)} {\ueps,V_1,V_2}, \\
& \wt H^{(k)}_{\ueps}(A) := \condexpect {\abs{\{\wt\tau_{k-3}\le t\le \wt\tau_k: Z^{(k)}(t)-Z^{(k)}(\wt\tau_{k})\in A\}} \cdot \wt\eta_k \cdot (1-\wh\eta_k)} {\ueps,V_1,V_2}. 
\end{align*}
As the middle and last parts in \eqref{decomp} are independent the following bounds apply

\begin{align}
  \label{bound for A23}
  \begin{aligned}
    &\condprobab{\{\wh\eta_k  = 1\} \cap A^{(k)}_{3,2}}{\ueps ,V_1,V_2}\le 
    Cr^{-1}\left(\int_{\R^3}G^{(k)}_{\ueps} (B_{x, 2r}) \wh H^{(k)}_{\ueps} (dx)
    +\int_{\R^3}H^{(k)}_{\ueps} (B_{x, 3r}) \wh G^{(k)}_{\ueps} (dx)\right),\\
    &\condprobab{\{\wt\eta_k  = 1\} \cap \{\wh\eta_k  = 0\} \cap A^{(k)}_{3,2}}{\ueps ,V_1,V_2 } \le \\
    &\hskip36mm\le 
    Cr^{-1}\left(\int_{\R^3} G^{(k)}_{\ueps} (B_{x, 2r}) \wt H^{(k)}_{\ueps} (dx)
    +\int_{\R^3}H^{(k)}_{\ueps} (B_{x, 3r}) \wt G^{(k)}_{\ueps} (dx)\right).
\end{aligned}
\end{align}


By \eqref{G bound} there exists a constant $C<\infty$ such that

\begin{align}\label{F bound}
  G_{\ueps}^{(k)}(B_{x,2r}) \le CF(x), \qquad  \qquad H_{\ueps}^{(k)}(B_{x,2r})\le C F(x)
\end{align}
where $F:\R^3 \to \R_+$

\begin{align*}
  F(x) = r \{\abs{x} \le r \} +\frac{r^3}{\abs{x}^2}\{r < \abs{x} \le  1 \} +\frac{r^3}{\abs{x}}\{\abs{x} >1\} + re^{-c\abs{x}}\one\{\exists t>0 : B_{x,2r}  \cap tV_2\}\{\abs{x} >r\}.
\end{align*}
For simplicity we will only treat the first term on the right hand side in the second line of \eqref{bound for A23} (this is the most difficult), the other terms can be dealt with similarly. 






Since during the middle section of \eqref{decomp} one component of the velocity does not change sign we can conclude

\begin{align}
\label{grbooooooo-loc}
\begin{aligned}
&
\wh G^{(k)}_{\ueps} (B_{0,s}),\wt G^{(k)}_{\ueps} (B_{0,s})
\le 
Crs, 
&&
\wh H^{(k)}_{\ueps} (B_{0,s}) ,\wt H^{(k)}_{\ueps} (B_{0,s})
\le 
Crs, 
\end{aligned}
\end{align}
and

\begin{align}
\label{grboo-glob}
\begin{aligned}
& \wh G^{(k)}_{\ueps} (\R^3), \wt G^{(k)}_{\ueps} (\R^3)\le Cr, 
&& \wh H^{(k)}_{\ueps} (\R^3), \wt H^{(k)}_{\ueps} (\R^3) \le Cr.
\end{aligned}
\end{align}

First note that by \eqref{grbooooooo-loc}

\begin{align*}
  \int_{\abs{x}>r} re^{-c\abs{x}}\one\{\exists t>0 : B_{x,2r}  \cap tV_2\} \wt{H}_{\ueps}^{(k)}(dx) &\le Cr^2 \int_{\abs{x}>r} e^{-c\abs{x}}\one\{\exists t>0 : B_{x,2r}  \cap tV_2\} dx\\
  &\le Cr^4 \int_{t > r} e^{-c\abs{tV_2}}dt \le Cr^4
\end{align*}
and

\begin{align*}
  \int_{\abs{x}>1}  \frac{r^3}{\abs{x}} \wt{H}_{\ueps}^{(k)} (B_{x,2r}) \le Cr^4.
\end{align*}
Finally let $\wt{F}(u) = r \{u \le r \} +\frac{r^3}{u^2}\{r < u \le  1 \}$, then by applying integration by parts
\begin{align*}
  \int_{\{\abs{x}<1\}} \wt{F}(\abs{x}) \wt{H}_{\ueps}^{(k)}(dx) &\le C \int_0^1 \wt{F}(u) d\wt{H}_{\ueps}^{(k)}(B_{0,u})\\
            & = Cr^3 \wt{H}_{\ueps}^{(k)}(B_{0,1}) - C \int_0^1 \wt{H}_{\ueps}^{(k)}(B_{0,u}) \wt{F}'(u)du\\
            &\le Cr^4 + Cr^4 \int_r^1 u^{-2}du\\ 
            &\le Cr^4 +Cr^3.
\end{align*}
\eqref{A12-A32} follows by inserting these bounds into \eqref{bound for A23}.

\subsection{Proof of Theorem \ref{thm:main-technical} - concluded}

The proof of Theorem \ref{thm:main-technical} now follows the same lines as {\cite[Section 7]{lutsko-toth_19}} repeated here for completeness.

Let $\{t \mapsto Y(t)\}$ be a Markovian flight process. Let $\{t \mapsto Z(t)\}$ be a coupled forgetful process. We split $\{t \mapsto Z(t)\}$ into i.i.d legs $(Z_n(t) : 0 \le t \le \theta_n)$, each associated to an i.i.d pack $\varpi_n = \left(\gamma_n;  \{\wt\xi_{n,j}\}_{j=1}^{\gamma}, \{\wt\beta_{n,j}\}_{j=1}^{\gamma},\{\wt w_{n,j}\}_{j=1}^{\gamma}\right)$. In addition, to each leg $(Z_n(t):0\le t \le \theta_n)$ we associate a wind-tree process coupled to that leg $(\cX_n(t):0 \le t \le \theta_n)$. From these components we construct the concatenated auxilliary process

\begin{align}
\cX(t) = \sum_{k=1}^{\nu_t} \cX(\theta_n)+\cX_{\nu_t+1}(\{t\}).
\end{align} 
Note that $t \mapsto \cX(t)$ is \emph{not} a physical process. Each leg is independent of the others. Finally let $t \mapsto X(t)$ be the true wind-tree process, coupled to $t \mapsto Y(t)$ and $ t \mapsto Z(t)$ as in Section \ref{ss:Joint Construction}.

We will use Propositions \ref{prop:intra-leg} and \ref{prop:inter-leg - beyond} to prove that until time $T=T(r)=\ordo(r^{-2})$ the processes $t\mapsto X(t)$, $t\mapsto \cX(t)$, and $t\mapsto Z(t)$ coincide with high probability. 

For this define the (discrete) stopping times
\begin{align*}
  & \rho := \min\{n: \cX_n(t)\not\equiv Z_n(t), 0\le t \le \theta_n\} \\
  & \sigma := \min\{n: \max\{\one_{\wt W_n}, \one_{\wh W_n}>0\}=1\}, 
\end{align*}
and note that by construction
\begin{align*}
\inf\{t: Z(t)\not= X(t)\}\ge\Theta_{\min\{\rho,\sigma\}-1}.
\end{align*}

\begin{lemma}
\label{lem: Theta is large}
Let $T=T(r)$ such that $\lim_{r\to\infty} T(r) =\infty$ and $\lim_{r\to\infty} r^{2} T(r) =0$. Then 
\begin{align}
\label{Theta is large}
\lim_{r\to 0}
\probab{\Theta_{\min\{\rho,\sigma\}-1}<T}=0. 
\end{align}
\end{lemma}

\begin{lemma}
\label{lem: Z is close to Y}
Let $T=T(r)$ such that $\lim_{r\to\infty} T(r) =\infty$ and $\lim_{r\to\infty} r^{2} T(r) =0$. Then for any $\delta>0$
\begin{align}
\label{Z is close to Y}
\lim_{r\to 0}
\probab{\max_{0\le t \le T}\abs{Y(t)-Z(t)}>\delta \sqrt{T}}=0. 
\end{align}
\end{lemma}


\begin{proof}
[Proof of Lemma \ref{lem: Theta is large}]
\begin{align}
\notag
\probab{\Theta_{\min\{\rho,\sigma\}-1}<T}
&
\le 
\probab{\rho\le 2 \expect{\theta}^{-1}T}
+
\probab{\sigma\le 2 \expect{\theta}^{-1}T}
+
\probab{\sum_{j=1}^{2 \expect{\theta}^{-1}T} \theta_j<T}
\\
\label{trivi-1}
&
\le
Cr^2 T
+
Cr^2 T
+
C e^{-c T}, 
\end{align}
where $C<\infty$ and $c>0$. The first term on the right hand side of \eqref{trivi-1} is bounded by union bound and \eqref{intra-exc} from Proposition \ref{prop:intra-leg}. Likewise the second term is bounded  by union bound  Proposition \ref{prop:inter-leg - beyond}. In bounding the third term we use a  large deviation upper bound for the sum of independent $\theta_j$-s. 

Finally \eqref{Theta is large} readily follows from \eqref{trivi-1}. 
\end{proof}

\begin{proof}
[Proof of Lemma \ref{lem: Z is close to Y}]
Note first that 
\begin{align*}
\max_{0\le t \le T}\abs{Y(t)-Z(t)}
\le 
\sum_{j=1}^{\nu_T+1}\eta_j\left( \sum_{i=j}^{\gamma_{\nu^\prime_j}}\xi_i\right), 
\end{align*}
with $\nu_T$ and $\eta_j$ defined in \eqref{tau and nu for Y}, respectively, \eqref{eta} and $\nu^{\prime}_j$ is $\nu_j$ from \eqref{Concat labels} (the label of the leg containing $j$). Hence, 
\begin{align}
  \notag
  \probab{\max_{0\le t \le T}\abs{Y(t)-Z(t)}>\delta \sqrt{T}}
  &\le 
  \probab{\sum_{j=1}^{2T}\eta_j\left( \sum_{i=j}^{\gamma_{\nu^\prime_j}}\xi_i\right)>\delta \sqrt{T}}+\probab{\nu_T>2T}\\
  \label{trivi-2}
  &\le C\delta^{-1}\sqrt{T}r + e^{-cT}, 
\end{align}
with $C<\infty$ and $c>0$. 
The first term on the right hand side of \eqref{trivi-2} is bounded by Markov's inequality and the bound
\begin{align*}
\expect{\eta_j \left( \sum_{i=j}^{\gamma_{\nu^\prime_j}}\xi_i\right)}\le C r. 
\end{align*} 
To see this recall the exponential tail bound for $\gamma$ \eqref{gammabnd}. The bound on the second term  follows from a straightforward large deviation estimate on $\nu_T\sim POI(T)$. 

Finally \eqref{Z is close to Y} readily follows from \eqref{trivi-2}. 

\end{proof}

\noindent \eqref{X to Y} is a direct consequence of Lemmas \ref{lem: Theta is large} and \ref{lem: Z is close to Y} and this concludes the proof of  Theorem \ref{thm:main-technical}. 
\qed

\section*{Acknowledgements}

The work of BT was supported by EPSRC (UK) Fellowship EP/P003656/1 and by NKFI (HU) K-129170. CL was supported by EPSRC Studentship EP/N509619/1 1793795. We would like to thank Jens Marklof for helping identify some of the relevant literature.

\vskip2cm

\hbox{
\hskip9cm
\vbox{\hsize=7cm\noindent
{\sc Authors' address:}
\\
School of Mathematics
\\
University of Bristol
\\
Bristol, BS8 1TW
\\
United Kingdom
\\
{\tt chris.lutsko@bristol.ac.uk}
\\
{\tt balint.toth@bristol.ac.uk}
}
}

\end{document}